%% file: arxiv_v2.tex
\documentclass[a4paper, 11pt, twoside]{article}

\input{macros}

\usepackage{fullpage}
\usepackage{amsmath,amsfonts,amssymb,amsthm,array}
\usepackage{algorithm}
\usepackage{caption}

\usepackage{graphicx}
\usepackage{booktabs} 
\usepackage{float}

\usepackage[capitalise]{cleveref}

\usepackage{algorithm}
\usepackage{algpseudocode}

\usepackage{enumitem}
\usepackage{nicefrac}

\usepackage{tikz}

\theoremstyle{plain}

\theoremstyle{remark}

\usepackage{mdframed} 
\usepackage{thmtools}

\usepackage[flushleft]{threeparttable} 

\usepackage{caption}
\usepackage{multirow}
\usepackage{colortbl}
\definecolor{bgcolor}{rgb}{0.8,1,1}
\definecolor{bgcolor2}{rgb}{0.8,1,0.8}
\definecolor{niceblue}{rgb}{0.0,0.19,0.56}

\usepackage{hyperref}
\hypersetup{colorlinks,linkcolor={blue},citecolor={niceblue},urlcolor={blue}}

\usepackage[numbers]{natbib}
\usepackage{sidecap}

\definecolor{shadecolor}{gray}{0.9}
\declaretheoremstyle[
headfont=\normalfont\bfseries,
notefont=\mdseries, notebraces={(}{)},
bodyfont=\normalfont,
postheadspace=0.5em,
spaceabove=1pt,
mdframed={
  skipabove=8pt,
  skipbelow=8pt,
  hidealllines=true,
  backgroundcolor={shadecolor},
  innerleftmargin=4pt,
  innerrightmargin=4pt}
]{shaded}

\usepackage{xspace}

\renewcommand*{\thefootnote}{\fnsymbol{footnote}}

\begin{document}
\title{\bf Accelerated Minimax Algorithms Flock Together\footnote{Published in SIAM Journal on Optimization.}}

\renewcommand*{\thefootnote}{\arabic{footnote}}
\setcounter{footnote}{0}

\author{TaeHo Yoon\textsuperscript{$\dagger$} \\
    \texttt{tyoon7@jhu.edu}  \and  Ernest K.\ Ryu\textsuperscript{$\ddagger$} \\
    \texttt{eryu@math.ucla.edu} }
\date{
\textsuperscript{$\dagger$}Department of Applied Mathematics \& Statistics, Johns Hopkins University \\
\textsuperscript{$\ddagger$}Department of Mathematics, University of California, Los Angeles
}
    
\maketitle

\begin{abstract}
Several new accelerated methods in minimax optimization and fixed-point iterations have recently been discovered, and, interestingly, they rely on a mechanism distinct from Nesterov's momentum-based acceleration. In this work, we show that these accelerated algorithms exhibit what we call the merging path (MP) property; the trajectories of these algorithms merge quickly. Using this novel MP property, we establish point convergence of existing accelerated minimax algorithms and derive new state-of-the-art algorithms for the strongly-convex-strongly-concave setup and for the prox-grad setup.
\end{abstract}

\section{Introduction}
\label{sec:introduction}

Minimax optimization problems of the form
\begin{align}
\label{eqn:minimax_problem}
    \underset{x \in \cX}{\textrm{minimize}} \,\, \underset{y \in \cY}{\textrm{maximize}} \,\,\,\, \lagrange (x,y)
\end{align}
have recently received increased attention in many machine learning applications including robust optimization \cite{SinhaNamkoongDuchi2018_certifying}, adversarial training \cite{MadryMakelovSchmidtTsiprasVladu2018_deep}, fair machine learning \cite{EdwardsStorkey2016_censoring}, and GANs \cite{GoodfellowPouget-AbadieMirzaXuWarde-FarleyOzairCourvilleBengio2014_generative}.
Consequently, there has been a large body of research on efficient algorithms for solving minimax optimization.
For deterministic smooth convex-concave objectives with Lipschitz continuous gradient, the classical extragradient (EG) \cite{Korpelevich1976_extragradient} and the optimistic gradient (OG) descent \cite{Popov1980_modification, RakhlinSridharan2013_online, DaskalakisIlyasSyrgkanisZeng2018_training} methods converge with $\mathcal{O}(1/k)$ rate on the squared gradient norm \cite{RyuYuanYin2019_ode, GorbunovLoizouGidel2022_extragradient}.
On the other hand, several recently discovered accelerated algorithms \cite{YoonRyu2021_accelerated, LeeKim2021_fast, Tran-DinhLuo2021_halperntype} based on the \emph{anchoring} mechanism exhibits the faster (and optimal) $\cO(1/k^2)$ convergence under the same setup.

In a different line of work on the fixed-point problem
\begin{align}
\label{eqn:fixed_point_problem}
\begin{array}{cc}
    \underset{z\in\reals^d}{\textrm{find}} & z = \opT(z) ,
\end{array}
\end{align}
the Halpern iteration \cite{Halpern1967_fixed}, which averages the iterates with the initial point, was shown to achieve an accelerated $\mathcal{O}(1/k^2)$ rate on the squared fixed-point residual norm when executed with optimal parameters \cite{Lieder2021_convergence, Kim2021_accelerated}.
Interestingly, both the Halpern iteration and the accelerated minimax algorithms rely on the anchoring mechanism, but a precise theoretical explanation of such apparent similarity was, to the best of our knowledge, yet missing.

\begin{figure*}[ht]
\captionsetup[subfigure]{justification=centering}
    \centering
    \begin{subfigure}{0.48\linewidth}
    \centering
    \includegraphics[scale=0.4]{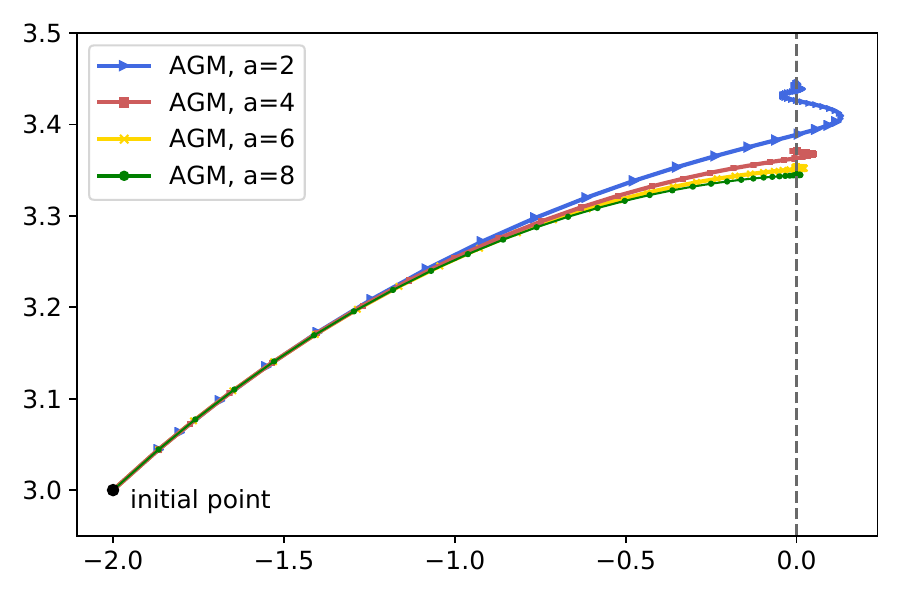}
    \caption{Divergent paths of AGM}
    \label{subfig:FISTA_divergent_paths}
    \end{subfigure}
    \hspace{.2cm}
    \begin{subfigure}{0.48\linewidth}
    \centering
    \includegraphics[scale=0.4]{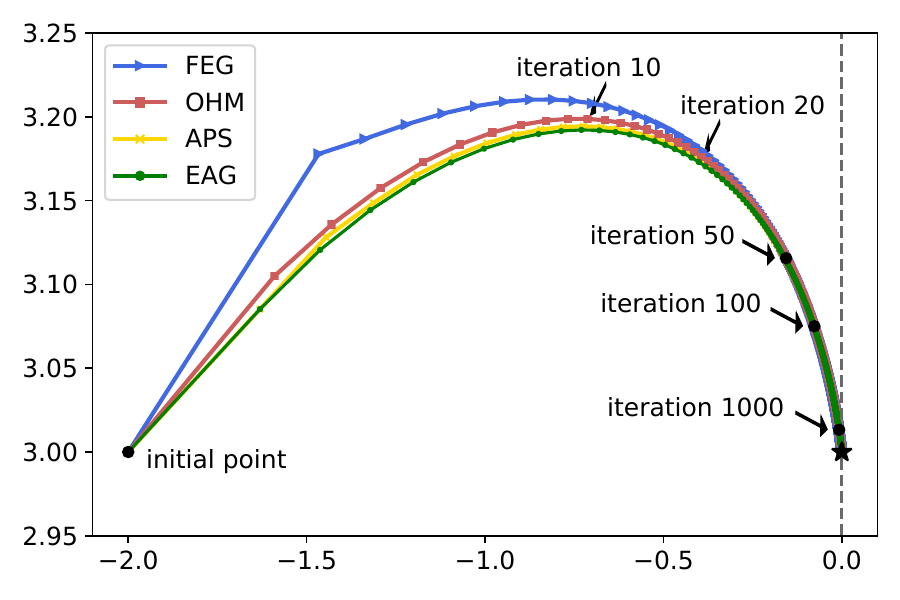}
    \caption{Merging paths of anchored methods}
    \label{subfig:anchored_merging_paths}
    \end{subfigure}
    \caption[]{Trajectories of \textbf{(left)} Nesterov's accelerated algorithms (AGM \cite{Nesterov1983_method, BeckTeboulle2009_fast, ChambolleDossal2015_convergence}; described in \cref{subsec:AGM_and_FISTA}) with $\alpha=0.025$ and distinct momentum parameters and \textbf{(right)} anchoring-based algorithms with $\alpha=0.1$. 
    Here, $\alpha$ denotes the step-size.
    On the right, paths quickly merge and become indistinguishable in 50 iterations.
     All algorithms are executed on the minimization problem with $f(x_1,x_2)=\frac{4x_1^2}{x_2}$ (which is convex and smooth on, e.g., $|x_1| \le 2, x_2 \ge 3$, where all iterates stay within), starting at $(x_1,x_2) = (-2,3)$.
    }
    \label{fig:algorithm_trajectories}
\end{figure*}

\subsection{Organization and contribution}
In this work, we identify the \emph{merging path (MP) property} among anchoring-based algorithms, which states that the trajectories of these algorithms quickly merge, and use this novel property to develop new accelerated algorithms. First, we show that known accelerated minimax algorithms \cite{YoonRyu2021_accelerated, LeeKim2021_fast, Tran-DinhLuo2021_halperntype} have $\mathcal{O}(1/k^2)$-merging paths with the optimized Halpern iteration \cite{Halpern1967_fixed, Lieder2021_convergence} and thereby establish their point convergence\footnote{By point convergence, we mean the iterates converge to a solution. For example, the $z_k$ iterates of EAG satisfy $z_k \to z_\star$, where $z_\star$ is a solution to~\cref{eqn:minimax_problem}.}. Second, we present a new accelerated minimax algorithm for smooth strongly-convex-strongly-concave setup designed to approximate, in the MP sense, the OC-Halpern method, which has an optimal accelerated rate. Third, we present a near-optimal accelerated proximal gradient type algorithm, designed to approximate, in the MP sense, the Halpern-accelerated Douglas-Rachford splitting \cite{DouglasRachford1956_numerical, LionsMercier1979_splitting}.

\subsection{The MP property}
\label{subsec:MP}

Let $\fA$ be a deterministic algorithm.
Write $\fA (x_0; \cP) = (x_0, x_1, x_2, \dots)$ to denote that $\fA$ applied to problem $\cP$ with starting point $x_0$ produces iterates $x_1,x_2,\dots$.
We say two algorithms $\fA_1 $ and $\fA_2$ have $\mathcal{O}(r(k))$-\emph{merging paths} if for any problem $\cP$ and $x_0 \in \reals^d$, the iterates $\fA_\ell (x_0; \cP) = \big( x_0, x_1^{(\ell)}, x_2^{(\ell)}, \dots \big)$ for $ \ell = 1, 2$ satisfy $\big\| x_k^{(1)} - x_k^{(2)} \big\|^2 = \mathcal{O}(r(k))$.
More concisely, we say $\fA_1$ and $\fA_2$ are $\mathcal{O}(r(k))$-MP if they have $\mathcal{O}(r(k))$-merging paths.

The MP property precisely formalizes a notion of near-equivalence of algorithms. In particular, it is a stronger notion of approximation compared to the conceptual ones of prior work \cite{MokhtariOzdaglarPattathil2020_unified, AhnSra2022_understanding}. 
Classical first-order algorithms for smooth convex minimization are not $\mathcal{O}(r(k))$-MP with $r(k)\rightarrow 0$. 
\cref{fig:algorithm_trajectories} shows that accelerated gradient methods with different momentum coefficients have divergent paths, even though their function values (of course) converge to the same optimal value.
In this paper, we show that the accelerated minimax algorithms for smooth convex-concave minimax optimization are $\mathcal{O}(1/k^2)$-MP (which is faster than the $o(1)$-point convergence). 
\cref{subfig:anchored_merging_paths} shows that the algorithms' paths indeed merge far before they converge to the limit.

\subsection{Related work}
The extragradient (EG) \cite{Korpelevich1976_extragradient} and Popov's algorithm \cite{Popov1980_modification}, also known as optimistic gradient (OG) descent \cite{RakhlinSridharan2013_online, DaskalakisIlyasSyrgkanisZeng2018_training}, have initiated the long stream of research in the optimization literature on the closely interconnected topics of minimax optimization \cite{Nemirovski2004_proxmethod, HeMonteiro2016_accelerated, KolossoskiMonteiro2017_accelerated, MertikopoulosLecouatZenatiFooChandrasekharPiliouras2019_optimistic, MokhtariOzdaglarPattathil2020_unified, MokhtariOzdaglarPattathil2020_convergence, AzizianMitliagkasLacoste-JulienGidel2020_tight, HamedaniAybat2021_primaldual, TsaknakisHongZhang2023_minimax}, variational inequalities \cite{Noor2003_new, Nesterov2007_dual, NesterovScrimali2011_solving, JuditskyNemirovskiTauvel2011_solving, CensorGibaliReich2011_strong, CensorGibaliReich2011_subgradient, MalitskySemenov2014_extragradient, Malitsky2015_projected, ChenLanOuyang2017_accelerated, HsiehIutzelerMalickMertikopoulos2019_convergence, Malitsky2020_golden}, and monotone inclusion problems \cite{SolodovSvaiter1999_hybrid, Tseng2000_modified, MonteiroSvaiter2010_complexity, MonteiroSvaiter2011_complexity, LyashkoSemenovVoitova2011_lowcost}.
A number of advances in minimax optimization have been made in the context of practical machine learning applications, including the multi-player game dynamics \cite{DaskalakisDeckelbaumKim2015_nearoptimal, SyrgkanisAgarwalLuoSchapire2015_fast, GidelHemmatPezeshkiPriolHuangLacoste-JulienMitliagkas2019_negative, HsiehAntonakopoulosMertikopoulos2021_adaptive} and GANs \cite{YadavShahXuJacobsGoldstein2018_stabilizing, GidelBerardVignoudVincentLacoste-Julien2019_variational, ChavdarovaGidelFleuretLacoste-Julien2019_reducing, LiangStokes2019_interaction, MertikopoulosLecouatZenatiFooChandrasekharPiliouras2019_optimistic, RyuYuanYin2019_ode, PengDaiZhangCheng2020_training}.
A recent trend of the field focuses on achieving near-optimal efficiency with respect to primal-dual gap or distance to the optimum for convex-concave minimax problems with distinct strong convexity, strong concavity or smoothness parameters for each variable \cite{ThekumparampilJainNetrapalliOh2019_efficient, LinJinJordan2020_nearoptimal, WangLi2020_improved, AlkousaGasnikovDvinskikhKovalevStonyakin2020_accelerated, YangZhangKiyavashHe2020_catalyst, TomininTomininBorodichKovalevGasnikovDvurechensky2023_accelerated, JinSidfordTian2022_sharper}, closely matching the lower bound results \cite{OuyangXu2021_lower, ZhangHongZhang2022_lower}.
Another line of works focuses on finding approximate stationary points of nonconvex-strongly-concave (or Polyak-Łojasiewicz) minimax problems, which can be viewed as reducing generalized versions of the objective's gradient magnitude \cite{BarazandehRazaviyayn2020_solving, LinJinJordan2020_gradient, KongMonteiro2021_accelerated, OstrovskiiLowyRazaviyayn2021_efficient, YangOrvietoLucchiHe2022_faster}.

The proximal point method (PPM) \cite{Moreau1965, Martinet1970_regularisation, Rockafellar1976_monotone} is an implicit algorithm for solving convex minimization, minimax optimization, and, more generally, monotone inclusion problems.
Although PPM is not always directly applicable, it has been used as the conceptual basis for developing many algorithms with approximate proximal evaluations \cite{Rockafellar1970_monotone, Guler1992_new, SolodovSvaiter1999_hybrid, Nemirovski2004_proxmethod, HeYuan2012_convergence, ChambollePock2016_ergodic, Drusvyatskiy2017_proximal, DavisGrimmer2019_proximally, KongMonteiro2021_accelerated, OstrovskiiLowyRazaviyayn2021_efficient, RafiqueLiuLinYang2021_weaklyconvexconcave}.
Several prior works \cite{MokhtariOzdaglarPattathil2020_unified, AhnSra2022_understanding} interpreted prominent forward algorithms such as Nesterov's accelerated gradient method (AGM), EG, and OG as conceptual approximations of PPM.

Since the seminal work of Nesterov \cite{Nesterov1983_method}, acceleration of first-order algorithms has been extensively studied for convex minimization problems \cite{Nesterov2005_smooth, BeckTeboulle2009_fast, KimFessler2016_optimized, B.VanScoyR.A.FreemanK.M.Lynch2018_fastest, ParkParkRyu2021_factorsqrt2, TaylorDrori2023_optimal}.
Lately, there has been a rapid progress in the discovery of accelerated algorithms with respect to the gradient magnitude of the objective for smooth convex problems \cite{KimFessler2021_optimizing, NesterovGasnikovGuminovDvurechensky2020_primalduala, DiakonikolasWang2022_potential, LeeParkRyu2021_geometric, ZhouTianSoCheng2022_practical} and smooth convex-concave minimax problems \cite{Diakonikolas2020_halpern, YoonRyu2021_accelerated, LeeKim2021_fast, Tran-DinhLuo2021_halperntype}.
Acceleration of PPM has been achieved for convex minimization \cite{Guler1992_new} and more recently for monotone inclusions in terms of the fixed-point residual \cite{Kim2021_accelerated, ParkRyu2022_exact}. Notably, the latter has been shown \cite{ParkRyu2022_exact, RyuYin2022_largescale} to be an instance of the Halpern iteration \cite{Halpern1967_fixed} with optimal parameters \cite{Lieder2021_convergence}.
While the accelerated minimax algorithms bear an apparent resemblance to the Halpern iteration \cite{Diakonikolas2020_halpern, Tran-DinhLuo2021_halperntype}, the latest results with optimal rates \cite{YoonRyu2021_accelerated, LeeKim2021_fast, Tran-DinhLuo2021_halperntype} were not formally understood through this connection.

The particular case of bilinearly-coupled minimax problems, $\lagrange(x,y) = f(x) + \langle x, Ay \rangle - g(y)$ for some convex functions $f$ and $g$ and a linear operator $A$, have been studied by many researchers.
A line of works \cite{Nesterov2005_smooth, XieHanZhang2021_dippa, ThekumparampilHeOh2022_lifted} considered the setups where $f$ and $g$ are smooth and one can only access their gradients.
Another line of works, more closely connected to the setup we study in \Cref{sec:composite_minimax}, considered the cases when $f$ or $g$ is proximable, and used primal-dual splitting algorithms to solve those problems \cite{ChambollePock2011_FirstOrderPrimalDual, Condat2013_primaldual, Vu2013_splitting, ChenLanOuyang2014_optimal, HeMonteiro2016_accelerated, KolossoskiMonteiro2017_accelerated, Yan2018_new, RyuYin2022_largescale}.
Primal-dual problem setups can often be reformulated into a more flexible format of composite monotone inclusion problems and addressed via forward-backward splitting algorithms \cite{Tseng2000_modified, CombettesPesquet2011_proximal, Condat2013_primaldual, Malitsky2015_projected, CsetnekMalitskyTam2019_shadow, MalitskyTam2020_forwardbackward}, which also extend the prox-grad-type methods including the iterative shrinkage-thresholding algorithms (ISTA) \cite{DaubechiesDefriseDeMol2004_iterative, BeckTeboulle2009_fast}.
Our result of \Cref{sec:composite_minimax} can be understood as acceleration of forward-backward algorithms with respect to the forward-backward residual, which generalizes the gradient magnitude.
For the special case of the projected-gradient setup, a near-optimal acceleration has been achieved in prior work \cite{Diakonikolas2020_halpern}.

\section{Background and preliminaries}
\label{sec:definitions}
A function $\lagrange(\cdot, \cdot): \cX \times \cY \to \reals$ is \emph{convex-concave} if $\cX$ and $\cY$ are convex sets, $\lagrange(x, y)$ is convex as a function of $x$ for all $y$, and $\lagrange(x, y)$ is concave as a function of $y$ for all $x$.
A point $(x_\star, y_\star) \in \cX \times \cY$ is a \emph{saddle point}, or \emph{(minimax) solution} of~\eqref{eqn:minimax_problem} if $\lagrange(x_\star, y) \le \lagrange(x_\star, y_\star) \le \lagrange(x, y_\star)$ for all $(x,y) \in \cX \times \cY$.

An operator $\opA$ on $\reals^d$, denoted $\opA\colon \reals^d \rightrightarrows \reals^d$, is a set-valued function, i.e., $\opA(z)\subseteq\reals^d$ for all $z \in \reals^d$.
For simplicity, we often write $\opA z = \opA(z)$.
If $\opA z$ contains exactly one element for all $z\in\reals^d$, then we say $\opA$ is \emph{single-valued} and view it as a function.
An operator $\opA\colon \reals^d \rightrightarrows \reals^d$ is \emph{monotone} if
\begin{align*}
    \langle \opA z - \opA z', z - z' \rangle \ge 0, \qquad \forall \, z,z'\in\reals^d ,
\end{align*}
where the notation means that $\langle u - v, z - z' \rangle \ge 0$ for any $u \in \opA z$ and $v \in \opA z'$.
For $\mu > 0$, an operator $\opA$ is $\mu$-\emph{strongly monotone} if
\begin{align*}
    \langle \opA z - \opA z', z - z' \rangle \ge \mu \|z-z'\|^2, \qquad \forall \, z,z'\in\reals^d .
\end{align*}
The graph of $\opA$ is denoted $\gra \opA = \{(z, u)\,|\, u \in \opA z\}$.
A monotone operator $\opA$ is \emph{maximally monotone} if there is no monotone operator $\opA'$ such that $\gra \opA \subset \gra \opA'$ strictly.
In particular, if $\opA$ is monotone, single-valued, and continuous as a function, then it is maximally monotone \cite[Corollary~20.28]{BauschkeCombettes2017_convex}.
Define
\begin{align*}
    (\opA+\opB)(z) = \{u+v\in\reals^d\,|\,u \in \opA z, v \in \opB z\}, \quad (\alpha\opA)(z) = \{\alpha u\in\reals^d\,|\,u \in \opA z\} 
\end{align*}
for operators $\opA$ and $\opB$ and scalar $\alpha\in\reals$.
The inverse of an operator $\opA$ is the operator $\opA^{-1}\colon \reals^d \rightrightarrows \reals^d$ defined by $\opA^{-1}(u) = \{z\in\reals^d\,|\,u \in \opA z\}$. 

In a \emph{monotone inclusion problem}, 
\begin{align}
\label{eqn:monotone_inclusion}
    \begin{array}{ll}
        \underset{z \in \reals^{d}}{\textrm{find}} & 0 \in \opA (z) ,
    \end{array}
\end{align}
where $\opA\colon \reals^d \rightrightarrows \reals^d$ is monotone, one finds zeros of an operator $\opA$.
For any $\alpha > 0$, we have $0 \in \opA (z) \iff z \in (\opI + \alpha\opA)(z) \iff z \in (\opI + \alpha\opA)^{-1}(z)$.
Therefore, \eqref{eqn:monotone_inclusion} is equivalent to the \emph{fixed-point problem}~\eqref{eqn:fixed_point_problem} with $\opT = (\opI + \alpha\opA)^{-1}=\JA$.
We call $\JA$ the \emph{resolvent} of $\alpha\opA$.

We say $\opT\colon \reals^d \to \reals^d$ is \emph{nonexpansive} if it is $1$-Lipschitz continuous and \emph{contractive} if it is $\rho$-Lipschitz with $\rho < 1$.
If $\opA$ is maximally monotone, then $\JA$ has full domain \cite{Minty1962_monotone} and is nonexpansive \cite{Rockafellar1976_monotone} for any $\alpha > 0$.
Denoting $\zer \opA = \{z\in\reals^d\,|\, 0\in \opA(z)\}$ and $\fix \opT = \{z\in\reals^d\,|\, z=\opT(z)\}$, we have $\zer\opA = \fix\JA$.
We call the algorithm $z_{k+1} = \JA(z_k)$ the \emph{proximal point method (PPM)} for $\opA$.

For differentiable and convex-concave $\lagrange\colon\reals^n \times \reals^m\rightarrow\reals$, we define
\[
    \sop\lagrange (x,y) = (\nabla_x \lagrange (x,y), -\nabla_y \lagrange (x,y))
\]
and call it the \emph{saddle operator} of $\lagrange$.
Then $\sop \lagrange$ is monotone \cite{Rockafellar1970_monotone}, and $z_\star = (x_\star, y_\star)$ is a solution of~\eqref{eqn:minimax_problem} if and only if $\sop \lagrange (z_\star) = 0$.
Therefore, a convex-concave minimax optimization problem~\eqref{eqn:minimax_problem} can be reformulated as a monotone inclusion problem \eqref{eqn:monotone_inclusion} with $\opA = \sop\lagrange$ and as a fixed-point problem \eqref{eqn:fixed_point_problem} with $\opT = \opJ_{\sop\lagrange}$.

The \emph{fixed-point residual norm} $\|z-\opT z\|$ quantifies the rate of convergence of an algorithm for finding fixed points.
When $\opT = \JB$, we have $z-\opT z = z - \JB (z) \in \alpha \opB (\JB (z))$, because
\begin{align*}
    u = \JB(z) = (\opI + \alpha\opB)^{-1}(z) \iff z \in (\opI + \alpha\opB)(u) = u + \alpha \opB u .
\end{align*}
Thus, when $\opT = \JB$ and $\opB$ is single-valued, a convergence rate on $\|z-\opT z\|$ is equivalent to a convergence rate on $\|\opB(\cdot)\|$.

\section{Accelerated smooth minimax algorithms are \texorpdfstring{$\mathcal{O}(1/k^2)$}{O(1/k2)}-MP}
\label{sec:smooth_minimax}

In this section, we prove the $\mathcal{O}(1/k^2)$-MP property among accelerated minimax algorithms and the optimized Halpern iteration.
The problem setup is unconstrained minimax problems (so $\cX=\reals^n$ and $\cY=\reals^m$) with convex-concave $\lagrange$ such that $\sop\lagrange$ is $L$-Lipschitz.
For notational convenience, write $\opB = \sop\lagrange \colon \reals^d \to \reals^d$, with $d=n+m$.
Our MP result provides new insight into these acceleration mechanisms and allows us to establish their point convergence $z_k \to z_\star \in \zer\opB$.

\subsection{Preliminaries: Accelerated minimax and proximal algorithms}
Recently, several accelerated minimax algorithms with rate $\|\opB z_k\|^2 = \mathcal{O}(1/k^2)$ have been proposed.
The list includes the \emph{extra anchored gradient (EAG)} algorithm \cite{YoonRyu2021_accelerated}
\begin{align*}
    z_{k+1/2} & = \beta_k z_0 + (1-\beta_k) z_k - \alpha \opB z_k \\
    z_{k+1}   & = \beta_k z_0 + (1-\beta_k) z_k - \alpha \opB z_{k+1/2},
\end{align*}
the \emph{fast extragradient (FEG)} algorithm \cite{LeeKim2021_fast}
\begin{align*}
    z_{k+1/2} & = \beta_k z_0 + (1-\beta_k) \left( z_k - \alpha \opB z_k \right) \\
    z_{k+1}   & = \beta_k z_0 + (1-\beta_k) z_k - \alpha \opB z_{k+1/2} ,
\end{align*}
and the \emph{anchored Popov's scheme (APS)} \cite{Tran-DinhLuo2021_halperntype}, with $v_0=z_0$,
\begin{align*}
    v_{k+1} & = \beta_k z_0 + (1-\beta_k) z_k - \alpha \opB v_k \\
    z_{k+1} & = \beta_k z_0 + (1-\beta_k) z_k - \alpha \opB v_{k+1} .
\end{align*}
Prior to the development of these minimax algorithms, the \emph{optimized Halpern's method (OHM)} \cite{Lieder2021_convergence, Kim2021_accelerated} was presented for solving the fixed-point problem~\eqref{eqn:fixed_point_problem} with nonexpansive operator $\opT\colon \reals^d \to \reals^d$:
\begin{align}
\begin{aligned}
    w_{k+1/2} & = \beta_k w_0 + (1-\beta_k) w_k \\
    w_{k+1}   & = \opT (w_{k+1/2}),
\end{aligned}
    \label{eqn:OHM}
\end{align}
where $\beta_k = \frac{1}{k+1}$.
OHM converges with rate $\|w_{k+1/2} - \opT w_{k+1/2}\|^2 \le 4\|w_0-w_\star\|^2/(k+1)^2$, if a fixed point $w_\star$ exists.

To clarify, the prior works \cite{YoonRyu2021_accelerated, Tran-DinhLuo2021_halperntype} study further general forms of EAG and APS with iteration-dependent step-sizes.
For the sake of simplicity, we consider the case of constant step-sizes in this paper.

\subsection{MP property among accelerated algorithms}
The accelerated forward algorithms, EAG, FEG, and APS, do not merely resemble OHM in their form. The algorithms are, in fact, near-equivalent in the sense that they have quickly merging paths.

\begin{theorem}[EAG$\approx$FEG$\approx$APS$\approx$OHM]
\label{thm:EAG_FEG_Halpern_MP}
Let $\opB\colon \reals^d \to \reals^d$ be monotone and $L$-Lipschitz and assume $z_\star \in\zer\opB = \fix \JB$ exists.
Let  $\opT=\JB$ and $\beta_k = \frac{1}{k+1}$.
Then FEG and OHM are $\mathcal{O}(\|z_0-z_\star\|^2/k^2)$-MP for $\alpha \in \left(0, \frac{1}{L}\right)$,
and there exists $\eta > 0$ (not depending on $L$) such that EAG, APS, and OHM are $\mathcal{O}(\|z_0-z_\star\|^2/k^2)$-MP for $\alpha \in \left(0, \frac{\eta}{L}\right)$.
\end{theorem}

\begin{proof}
Here we provide the proof for the FEG case and defer the proofs for EAG and APS to Appendix~\ref{sec:EAG_APS_MP_proof}.
Let $\{w_k\}_{k=0,1,\dots}$ be the OHM iterates and $\{z_k\}_{k=0,1,\dots}$ the FEG iterates, and assume $z_0=w_0$.
Then
\begin{align*}
    z_{k+1} - w_{k+1} & = \left( \beta_k z_0 + (1-\beta_k) z_k - \alpha \opB z_{k+1/2} \right) - \left( \beta_k w_0 + (1-\beta_k) w_k - \alpha \opB w_{k+1} \right) \\
    & = (1-\beta_k) (z_k - w_k) + \alpha \left( \opB w_{k+1} - \opB z_{k+1/2} \right)
\end{align*}
where we used $(\opI + \alpha \opB) (w_{k+1}) = w_{k+1/2}$ in the first equality and $z_0 = w_0$ in the second.
Thus,
\begin{align}
\label{eqn:FEG_MP_identity}
    \begin{aligned}
    \left\| z_{k+1} - w_{k+1} \right\|^2 & = (1-\beta_k)^2 \left\|z_k - w_k\right\|^2 + 2 \left\langle \alpha (1-\beta_k) (z_k - w_k), \opB w_{k+1} - \opB z_{k+1/2} \right\rangle \\
    & \quad\quad + \alpha^2 \left\| \opB w_{k+1} - \opB z_{k+1/2} \right\|^2 .
    \end{aligned}
\end{align}
Now use the similar identity
\begin{align*}
    z_{k+1/2} - w_{k+1} = (1-\beta_k) (z_k - w_k) - \alpha (1-\beta_k) \opB z_k + \alpha \opB w_{k+1}
\end{align*}
to rewrite the inner product term in~\eqref{eqn:FEG_MP_identity} as
\begin{align*}
    & 2 \left\langle \alpha (1-\beta_k) (z_k - w_k), \opB w_{k+1} - \opB z_{k+1/2} \right\rangle \\
    & = 2 \left\langle \alpha (z_{k+1/2} - w_{k+1}) + \alpha^2 (1-\beta_k) \opB z_k - \alpha^2 \opB w_{k+1} , \opB w_{k+1} - \opB z_{k+1/2} \right\rangle \\
    & \le 2\alpha^2 \left\langle (1-\beta_k) \opB z_k - \opB w_{k+1}, \opB w_{k+1} - \opB z_{k+1/2} \right\rangle , 
\end{align*}
where the last inequality follows from monotonicity of $\opB$.
Combining this with~\eqref{eqn:FEG_MP_identity}, we obtain
\begin{align*}
    \left\| z_{k+1} - w_{k+1} \right\|^2 & \le (1-\beta_k)^2 \left\|z_k - w_k\right\|^2 + 2\alpha^2 \left\langle (1-\beta_k) \opB z_k - \opB w_{k+1}, \opB w_{k+1} - \opB z_{k+1/2} \right\rangle \\
    & \quad + \alpha^2 \left\| \opB w_{k+1} - \opB z_{k+1/2} \right\|^2 \\
    & = (1-\beta_k)^2 \left\|z_k - w_k\right\|^2 \underbrace{ -\, \alpha^2 \left\| \opB w_{k+1} \right\|^2  + 2\alpha^2 (1-\beta_k) \langle \opB z_k, \opB w_{k+1} \rangle}_{\le \alpha^2 (1-\beta_k)^2 \|\opB z_k\|^2} \\
    & \quad - 2\alpha^2 (1-\beta_k) \langle \opB z_k, \opB z_{k+1/2} \rangle + \alpha^2 \|\opB z_{k+1/2}\|^2 .
\end{align*}
Plugging in $\beta_k = \frac{1}{k+1}$ and multiplying both sides by $(k+1)^2$, we obtain
\begin{align*}
    (k+1)^2 \left\| z_{k+1} - w_{k+1} \right\|^2 \le & \, k^2 \left\| z_k - w_k \right\|^2 + \alpha^2 \|k \opB z_k - (k+1) \opB z_{k+1/2}\|^2 .
\end{align*}
The conclusion then follows from the following \cref{lem:FEG_summability}, whose proof is deferred to \Cref{sec:FEG_summability_lemma_proof}.
\end{proof}

\begin{lemma}
\label{lem:FEG_summability}
Let $\opB\colon \reals^d \to \reals^d$ be monotone and $L$-Lipschitz and assume $z_\star \in\zer\opB$ exists.
Let $\{z_k\}_{k=0,1,\dots}$ be the iterates of FEG with $\beta_k = \frac{1}{k+1}$.
For $\alpha \in \left(0, \frac{1}{L}\right)$, 
\begin{align*}
    \sum_{k=0}^\infty \|k \opB z_k - (k+1) \opB z_{k+1/2}\|^2 \le \frac{1}{\alpha^2 (1-\alpha^2 L^2)} \|z_0 - z_\star\|^2 < \infty .
\end{align*}
\end{lemma}

As it is well known \cite[Theorem~30.1]{BauschkeCombettes2017_convex} that OHM converges to 
\[
\proj_{\mathrm{Fix} \opT}(z_0) = \underset{z\in\fix\opT}{\argmin} \|z-z_0\|,
\]
\cref{thm:EAG_FEG_Halpern_MP} immediately implies the point convergence of EAG, FEG, and APS.

\begin{corollary}
\label{cor:point-conv}
In the setup of \cref{thm:EAG_FEG_Halpern_MP}, iterates of EAG, FEG, and APS converge to $\proj_{\zer \opB}(z_0)$.
\end{corollary}

\section{Fastest rate for unconstrained smooth strongly-convex-strongly-concave minimax optimization with respect to gradient norm}
\label{sec:scsc_minimax}
In \Cref{sec:smooth_minimax}, we identified that prior accelerated minimax algorithms approximate, in the MP sense, the optimal proximal algorithm OHM. In this section, we design a novel algorithm to approximate, in the MP sense, the optimal proximal algorithm OC-Halpern \cite{ParkRyu2022_exact} and thereby achieve the state-of-the-art rate\footnote{This word ``fastest'' in the section title means ``fastest known'', and does not mean the rate cannot be improved further.} for minimax problems with unconstrained $L$-smooth $\mu$-strongly-convex-strongly-concave $\lagrange$ (so $\sop\lagrange \colon \reals^d \to \reals^d$ is $L$-Lipschitz and $\mu$-strongly monotone). Later in \Cref{sec:composite_minimax}, we repeat the strategy of designing an algorithm to approximate, in the MP sense, a known accelerated proximal algorithm.

We design the novel algorithm \emph{\textbf{S}trongly \textbf{M}onotone \textbf{E}xtra \textbf{A}nchored \textbf{G}radient\textbf{\texttt{+}}}\footnote{The ``\texttt{+}'' follows the nomenclature of \cite{DiakonikolasDaskalakisJordan2021_efficient}, where the two-time-scale EG was named EG\texttt{+}.} 
\begin{align}
\label{eqn:SM-EAG+}
\begin{aligned}
    z_{k+1/2} & = \beta_k z_0 + (1-\beta_k) z_k - \eta_k \alpha \opB z_k \\
    z_{k+1}   & = \beta_k z_0 + (1-\beta_k) z_k - \alpha \opB z_{k+1/2} ,
\end{aligned}
\tag{SM-EAG\texttt{+}}
\end{align}
where $\opB=\sop\lagrange$, $\beta_k = \frac{1}{\sum_{j=0}^k (1+2\alpha\mu)^j}$, $\eta_k = \frac{1-\beta_k}{1+2\alpha\mu}$, and $0 < \alpha \le \frac{\sqrt{L^2 + \mu^2} + \mu}{L^2}$, to approximate OC-Halpern in the MP sense. \ref{eqn:SM-EAG+} inherits the accelerated rate of OC-Halpern and achieves the fastest known rate in the setup with respect to the gradient norm, improving upon the prior rates of EG and OG by factors of 8 and 4, respectively.
(OC-Halpern is not a forward algorithm as it uses the proximal operator $\JB$, while \ref{eqn:SM-EAG+} is a forward algorithm using evaluations of $\opB$.)

\subsection{Preliminaries: Proximal algorithm with exact optimal complexity}
For the fixed-point problem \eqref{eqn:fixed_point_problem} with a $\gamma^{-1}$-contractive $\opT$, the recently presented \emph{OC-Halpern} \cite[Corollary~3.3, Theorem~4.1]{ParkRyu2022_exact},
which has the same form~\eqref{eqn:OHM} as OHM but with $\beta_k = \left(\sum_{j=0}^k \gamma^{2j}\right)^{-1}$, 
achieves the exact optimal accelerated rate 
of $\|w_{k+1/2} - \opT w_{k+1/2}\|^2 \le \left( 1 + \gamma^{-1} \right)^2 \left(\sum_{j=0}^k \gamma^j \right)^{-2} \|w_0 - w_\star\|^2$, where $w_\star \in \fix\opB$.
We consider OC-Halpern with $\opT = \JB$, which is contractive if $\opB$ is maximally monotone and $\mu$-strongly monotone \cite[Proposition~23.8, 23.13]{BauschkeCombettes2017_convex}.

\subsection{MP property and convergence analyses of SM-EAG\texttt{+}}
\ref{eqn:SM-EAG+} approximates OC-Halpern in the MP sense. 

\begin{theorem}[\ref{eqn:SM-EAG+}$\approx$OC-Halpern]
\label{thm:sm-EAG_OC-Halpern_MP}
Let $\opB\colon \reals^d \to \reals^d$ be $\mu$-strongly monotone and $L$-Lipschitz with $0 < \mu \le L$.
Let $z_\star$ be the zero of $\opB$.
If $\alpha \in \big( 0, \frac{\sqrt{L^2 + \mu^2} + \mu}{L^2} \big)$, then for any $\epsilon \in (0, 1)$, \ref{eqn:SM-EAG+} and OC-Halpern with $\opT = \JB$ and $\beta_k = \frac{1}{\sum_{j=0}^k (1+2\alpha\mu)^j}$ are $\mathcal{O}\left(\frac{\|z_0-z_\star\|^2}{(1+2\alpha\mu(1-\epsilon))^{k}}\right)$-MP.
\end{theorem}

To clarify, in \cref{thm:sm-EAG_OC-Halpern_MP}, we view $\JB$ as a $\gamma^{-1}$-contractive operator with $\gamma^{-1} = \frac{1}{\sqrt{1+2\alpha\mu}} > \frac{1}{1+\alpha\mu}$, where $\frac{1}{1+\alpha\mu}$ is the tight contraction factor for $\JB$ \cite[Proposition~23.13]{BauschkeCombettes2017_convex}, and then apply OC-Halpern. This slack, which is negligible in the regime $L \gg \mu$, is necessary for proving the MP result.

\ref{eqn:SM-EAG+} inherits the convergence rate of OC-Halpern, 
since the paths of OC-Halpern and \ref{eqn:SM-EAG+} merge at rate $\mathcal{O}\left(\frac{\|z_0-z_\star\|^2}{(1+2\alpha\mu(1-\epsilon))^{k}}\right)$ for any $\epsilon > 0$, arbitrarily close to the order of convergence $\|\alpha \opB w_{k+1}\|^2 = \mathcal{O}\left(\frac{\|w_0 - w_\star\|^2}{(1+2\alpha\mu)^k}\right)$ of OC-Halpern \cite{ParkRyu2022_exact}.
On the other hand, EG and OG can be viewed as approximations of the proximal point method (PPM) \cite{MokhtariOzdaglarPattathil2020_unified}, which is slower than OC-Halpern.
Additionally, it is unclear whether EG and OG exhibit an MP property to PPM fast enough to preserve PPM's convergence rate since EG and OG have no anchoring mechanism inducing the MP property.

\begin{theorem}[Fast rate of \ref{eqn:SM-EAG+}]
\label{thm:sm-EAG}
Let $\opB\colon \reals^d \to \reals^d$ be $\mu$-strongly monotone and $L$-Lipschitz with $0 < \mu \le L$.
Let $z_\star$ be the zero of $\opB$.
For $\alpha \in \big( 0, \frac{\sqrt{L^2 + \mu^2} + \mu}{L^2} \big]$, SM-EAG\texttt{+} exhibits the rate
\begin{align*}
    \|\opB z_k\|^2 \le \frac{\left(\sqrt{1+2\alpha\mu} + 1\right)^2}{\alpha^2 \left( \sum_{j=0}^{k-1} (1+2\alpha\mu)^{\frac{j}{2}} \right)^2} \|z_0 - z_\star\|^2 ,
    \qquad\text{for } k=1,2,\dots.
\end{align*}
\end{theorem}

With the choice $\alpha = \frac{\sqrt{L^2 + \mu^2} + \mu}{L^2}$, \cref{thm:sm-EAG} implies
\begin{align*}
    \|\opB z_k\|^2 = \mathcal{O} \left( e^{-2k \frac{\mu}{L}} L^2 \|z_0 - z_\star\|^2 \right) 
\end{align*}
when $L \gg \mu$.
Therefore, \ref{eqn:SM-EAG+} requires roughly $\frac{L}{\mu} \log \frac{L\|z_0-z_\star\|}{\epsilon}$ iterations to achieve $\|\opB z_k\| \le \epsilon$.
On the other hand, simultaneous gradient descent-ascent has the rate \cite[Proposition~26.16]{BauschkeCombettes2017_convex}
\begin{align*}
    \|\opB z_k\|^2 \le L^2\|z_k - z_\star\|^2 = \mathcal{O} \left( e^{-k\frac{\mu^2}{L^2}} L^2 \|z_0-z_\star\|^2 \right)
\end{align*}
when $L \gg \mu$, and therefore requires $\frac{2L^2}{\mu^2} \log \frac{L\|z_0-z_\star\|}{\epsilon} \gg \frac{L}{\mu} \log \frac{L\|z_0-z_\star\|}{\epsilon}$ iterations to achieve $\|\opB z_k\| \le \epsilon$.
EG and OG has the rate \cite{MokhtariOzdaglarPattathil2020_unified, AzizianMitliagkasLacoste-JulienGidel2020_tight}:
\begin{align*}
    \|\opB z_k\|^2 \le L^2\|z_k - z_\star\|^2 = \mathcal{O} \left( e^{-k\frac{\mu}{4L}} L^2\|z_0 - z_\star\|^2 \right) ,
\end{align*}
when $L \gg \mu$, and therefore requires $\frac{8L}{\mu} \log \frac{L\|z_0-z_\star\|}{\epsilon}$ iterations to achieve $\|\opB z_k\| \le \epsilon$.
Thus, \ref{eqn:SM-EAG+} is $8\times$ faster than EG and $4\times$ faster than OG (since OG makes a single call to the gradient oracle at each iteration).

Notably, when $\mu=0$, \ref{eqn:SM-EAG+} coincides with FEG \cite{LeeKim2021_fast}, which has the fastest known rate of $\|\opB z_k\|^2 \le \frac{4L^2 \|z_0-z_\star\|^2}{k^2}$ when $\opB$ is (non-strongly) monotone and $L$-Lipschitz.
This rate is precisely recovered by \cref{thm:sm-EAG} when $\mu = 0$, with $\alpha = \frac{1}{L}$.

Although OC-Halpern is an algorithm designed for exact optimal convergence with respect to the fixed-point residual norm $\|w_{k+1/2} - \opT w_{k+1/2}\|^2$, its iterates also converge linearly at a fast rate comparable to PPM: $\|w_k - w_\star\|^2 = \cO \left(\gamma^{-2k} \|w_0 - w_\star\|^2\right)$.
The proof of this fact is simple but new; we provide it in \cref{subsec:oc-halpern-linear-conv}.
Together with \cref{thm:sm-EAG_OC-Halpern_MP}, this implies that the point convergence of \smeag~also occurs at a fast rate of $\|z_k - z_\star\|^2 = \cO \left(e^{-2k \frac{\mu}{L}} \|z_0 - z_\star\|^2 \right)$ when $L \gg \mu$ and $\alpha$ is chosen appropriately.

\begin{theorem}
\label{thm:sm-eag-fast-point-convergence}
Let $\opB\colon \reals^d \to \reals^d$ be $\mu$-strongly monotone and $L$-Lipschitz with $0 < \mu \le L$. Let $z_\star$ be the zero of $\opB$.
For $\alpha \in \big( 0, \frac{\sqrt{L^2 + \mu^2} + \mu}{L^2} \big)$, \smeag~converges with the rate
\begin{align}
\label{eqn:smeag-dist-to-sol-rate}
    \|z_k - z_\star\|^2 & \le 2 \left[ \frac{(1 + \sqrt{2\alpha\mu})^2}{(1+2\alpha\mu)^{k+1}} + \left( \frac{\left( 1 + 2\alpha \mu \left(\epsilon^{-1} - 1\right) \right)}{(1+2\alpha\mu(1-\epsilon))^{k+1}} \frac{1+2\alpha\mu}{1+2\alpha\mu-\alpha^2 L^2} \right) \right] \|z_0 - z_\star\|^2 \\
    & = \cO \left( \frac{\|z_0 - z_\star\|^2}{(1+2\alpha\mu(1-\epsilon))^k} \right) \nonumber
\end{align}
where $\epsilon \in (0, 1)$ can be chosen arbitrarily.
\end{theorem}

\subsection{Proof of \texorpdfstring{\cref{thm:sm-EAG_OC-Halpern_MP}}{Theorem 4.1}}
Denote the iterates of OC-Halpern by $\{w_k\}_{k=0,1,\dots}$ and \ref{eqn:SM-EAG+} iterates by $\{z_k\}_{k=0,1,\dots}$.
Similarly as in \Cref{sec:smooth_minimax}, we have
\begin{align*}
    \left\| z_{k+1} - w_{k+1} \right\|^2 & = (1-\beta_k)^2 \left\|z_k - w_k\right\|^2 + 2 \left\langle \alpha (1-\beta_k) (z_k - w_k), \opB w_{k+1} - \opB z_{k+1/2} \right\rangle \\
    & \quad + \alpha^2 \left\| \opB w_{k+1} - \opB z_{k+1/2} \right\|^2 .
\end{align*}
Using $z_{k+1/2} - w_{k+1} = (1-\beta_k) (z_k-w_k) - \eta_k \alpha \opB z_k + \alpha \opB w_{k+1}$, we get
\begin{align}
    & 2 \left\langle \alpha (1-\beta_k) (z_k - w_k), \opB w_{k+1} - \opB z_{k+1/2} \right\rangle \nonumber \\
    & = 2 \left\langle \alpha(z_{k+1/2} - w_{k+1}) + \alpha^2 (\eta_k \opB z_k - \opB w_{k+1}) , \opB w_{k+1} - \opB z_{k+1/2} \right\rangle \nonumber \\
    & \le -2\alpha\mu \|z_{k+1/2} - w_{k+1}\|^2 + 2\alpha^2 \left\langle \eta_k \opB z_k - \opB w_{k+1} , \opB w_{k+1} - \opB z_{k+1/2} \right\rangle \label{eqn:sm-EAG_Halpern_MP_strong_monotonicity}
\end{align}
where the last line uses $\mu$-strong monotonicity of $\opB$.
Note that for any $\epsilon \in (0,1)$ and $x,y\in\reals^d$, we have
\begin{align*}
    \epsilon \|x\|^2 + \frac{1}{\epsilon} \|y\|^2 \ge 2 \langle x, y \rangle \implies \|x-y\|^2 \ge (1-\epsilon)\|x\|^2 - \left( \frac{1}{\epsilon} - 1\right) \|y\|^2 .
\end{align*}
Using the last inequality with $x = z_{k+1} - w_{k+1}$ and $y = z_{k+1} - z_{k+1/2} = \alpha(\eta_k \opB z_k - \opB z_{k+1/2})$, we can upper-bound~\eqref{eqn:sm-EAG_Halpern_MP_strong_monotonicity}, so that
\begin{align*}
    & \|z_{k+1} - w_{k+1}\|^2 \\
    & \le (1-\beta_k)^2 \|z_k - w_k\|^2 -2\alpha\mu \left( (1-\epsilon) \|z_{k+1} - w_{k+1}\|^2 - \left( \frac{1}{\epsilon} - 1\right) \alpha^2 \|\eta_k \opB z_k - \opB z_{k+1/2}\|^2 \right) \\
    & \quad + 2\alpha^2 \left\langle \eta_k \opB z_k - \opB w_{k+1} , \opB w_{k+1} - \opB z_{k+1/2} \right\rangle + \alpha^2 \left\| \opB w_{k+1} - \opB z_{k+1/2} \right\|^2 .
\end{align*}
Rearranging the last inequality, we obtain
\begin{align*}
    & (1 + 2\alpha\mu (1-\epsilon)) \|z_{k+1} - w_{k+1}\|^2 \\
    & \le (1-\beta_k)^2 \|z_k - w_k\|^2 + 2\alpha \mu \left( \frac{1}{\epsilon} - 1\right) \alpha^2 \|\eta_k \opB z_k - \opB z_{k+1/2}\|^2 \\
    & \quad \underbrace{-\alpha^2 \|\opB w_{k+1}\|^2 + 2\eta_k \alpha^2 \langle \opB w_{k+1}, \opB z_k \rangle}_{\le \eta_k^2 \alpha^2 \|\opB z_k\|^2} - 2\eta_k \alpha^2 \langle \opB z_k, \opB z_{k+1/2} \rangle + \alpha^2 \|\opB z_{k+1/2}\|^2 \\
    & \le (1-\beta_k)^2 \|z_k - w_k\|^2 + \left( 1 + 2\alpha \mu \left( \frac{1}{\epsilon} - 1\right) \right) \alpha^2 \|\eta_k \opB z_k - \opB z_{k+1/2}\|^2 .
\end{align*}
Multiplying $(1+2\alpha\mu(1-\epsilon))^k$ throughout the above equation and upper-bounding $(1-\beta_k)^2$ by $1$ we get
\begin{align*}
    & (1+2\alpha\mu(1-\epsilon))^{k+1} \|z_{k+1} - w_{k+1}\|^2 \\
    & \quad \quad \le (1+2\alpha\mu(1-\epsilon))^k \|z_k - w_k\|^2 + (1+2\alpha\mu)^k \left( 1 + 2\alpha \mu \left( \frac{1}{\epsilon} - 1\right) \right) \alpha^2 \|\eta_k \opB z_k - \opB z_{k+1/2}\|^2 \\
    & \quad \quad \le \cdots \le \left( 1 + 2\alpha \mu \left( \frac{1}{\epsilon} - 1\right) \right) \alpha^2 \sum_{j=0}^k (1+2\alpha\mu)^j \|\eta_j \opB z_j - \opB z_{j+1/2}\|^2 .
\end{align*}
Thus, provided that the following \cref{lem:sm-EAG_Lyapunov} (in particular~\eqref{eqn:sm-EAG_summability}) holds, we conclude
\begin{align*}
    (1+2\alpha\mu(1-\epsilon))^{k+1} \|z_{k+1} - w_{k+1}\|^2 \le \left( 1 + 2\alpha \mu \left( \frac{1}{\epsilon} - 1\right) \right) \frac{1+2\alpha\mu}{1+2\alpha\mu-\alpha^2 L^2} \|z_0 - z_\star\|^2 ,
\end{align*}
i.e., the two algorithms are $\mathcal{O}\left(\frac{\|z_0-z_\star\|^2}{(1+2\alpha\mu(1-\epsilon))^k}\right) = \mathcal{O}\left( e^{-2k\alpha\mu (1-\epsilon)}\|z_0-z_\star\|^2 \right)$-MP.

\begin{lemma}
\label{lem:sm-EAG_Lyapunov}
Let $\opB\colon \reals^d \to \reals^d$ be a $\mu$-strongly monotone, $L$-Lipschitz operator ($0 < \mu \le L$), and let
$\{z_k\}_{k=0,1,\dots}$ be the iterates of \ref{eqn:SM-EAG+} with $\beta_k = \frac{1}{\sum_{j=0}^k (1+2\alpha\mu)^j}$.
Then, the quantities
\begin{align*}
    V_k = p_k \|\opB z_k\|^2 + q_k \langle \opB z_k - \mu (z_k - z_0), z_k - z_0 \rangle + \left( \frac{1}{2\alpha} + \mu \right) \|z_0-z_\star\|^2
\end{align*}
with $p_0 = q_0 = 0$ and $q_k = \frac{1}{(1+2\alpha\mu)^{k-1}} \frac{1}{\beta_{k-1}}$, $p_k = \frac{\eta_k \alpha q_k}{2 \beta_k}$ for $k=1,2,\dots$ satisfy $V_k \ge 0$, $V_0 - V_1 \ge \frac{\alpha (1+2\alpha\mu-\alpha^2 L^2)}{2}\|\opB z_0\|^2$, and
\begin{align*}
    V_k - V_{k+1} \ge \frac{\alpha (1 + 2\alpha\mu - \alpha^2 L^2) q_k}{2\beta_k (1-\beta_k)} \|\eta_k \opB z_k - \opB z_{k+1/2}\|^2
\end{align*}
for $k=1,2,\dots$.
In particular, if $\alpha \in \big( 0, \frac{\sqrt{L^2 + \mu^2} + \mu}{L^2} \big]$, $V_k$ is a Lyapunov function for the algorithm, and if $\alpha \in \big( 0, \frac{\sqrt{L^2 + \mu^2} + \mu}{L^2} \big)$, we additionally have the summability result
\begin{align}
    \sum_{k=0}^\infty (1+2\alpha\mu)^k \|\eta_k \opB z_k - \opB z_{k+1/2}\|^2 \le \frac{1+2\alpha\mu}{\alpha^2 (1 + 2\alpha\mu - \alpha^2 L^2)} \|z_0 - z_\star\|^2  < \infty .
    \label{eqn:sm-EAG_summability}
\end{align}
\end{lemma}

We defer the proof of \cref{lem:sm-EAG_Lyapunov} to \Cref{sec:scsc_minimax_proofs}.

\subsection{Proof of \texorpdfstring{\cref{thm:sm-EAG}}{Theorem 4.2}}
Because the term $\left( \frac{1}{2\alpha}+\mu \right) \|z_0-z_\star\|^2$ subtracted from $V_k$ to obtain $W_k$ does not depend on $k$, clearly $W_k$ inherits nonincreasingness from $V_k$, i.e., $W_k \le \cdots \le W_0 = 0$.
Therefore, given that $q_k < \lambda < q_k + 4\mu p_k$, we have
\begin{align}
    & \frac{q_k \lambda \mu}{\lambda - q_k} \|z_0 - z_\star\|^2 \ge W_k + \frac{q_k \lambda \mu}{\lambda - q_k} \|z_0 - z_\star\|^2 \ge \left( p_k - \frac{\lambda - q_k}{4\mu} \right) \|\opB z_k\|^2 \nonumber \\
    & \implies \|\opB z_k\|^2 \le \left( p_k - \frac{\lambda - q_k}{4\mu} \right)^{-1} \frac{q_k \lambda \mu}{\lambda - q_k} \|z_0 - z_\star\|^2 . \label{eqn:sm-EAG_convergence_rate}
\end{align}
One can straightforwardly verify that 
\begin{align}
    \lambda_\star = (q_k (q_k + 4\mu p_k))^{1/2}
    \label{eqn:sm-EAG_lambda}
\end{align}
lies between $q_k$ and $q_k + 4\mu p_k$, and minimizes the factor $\left( p_k - \frac{\lambda - q_k}{4\mu} \right)^{-1} \frac{q_k \lambda \mu}{\lambda - q_k}$ in the right hand side of \eqref{eqn:sm-EAG_convergence_rate}.
Plugging in the explicit expressions $q_k = \frac{(1+2\alpha\mu)^k - 1}{2\alpha\mu (1+2\alpha\mu)^{k-1}}$ and $p_k = \frac{\alpha}{2}(1+2\alpha\mu)^{k-1} q_k^2 = \frac{((1+2\alpha\mu)^k - 1)^2}{8\alpha\mu^2 (1+2\alpha\mu)^{k-1}}$ (which uses \eqref{eqn:sm-EAG_pkqk}) into~\eqref{eqn:sm-EAG_lambda}, we obtain $\lambda_\star = (1+2\alpha\mu)^{\frac{k}{2}} q_k$ and
\begin{align*}
    \left( p_k - \frac{\lambda_\star - q_k}{4\mu} \right)^{-1} \frac{q_k \lambda_\star \mu}{\lambda_\star - q_k} & = \frac{4\mu^2}{((1+2\alpha\mu)^{\frac{k}{2}} - 1)^2} 
    = \frac{(\sqrt{1+2\alpha\mu}+1)^2}{\alpha^2 \left( \sum_{j=0}^{k-1} (1+2\alpha\mu)^{\frac{j}{2}} \right)^2} .
\end{align*}
Therefore, \eqref{eqn:sm-EAG_convergence_rate} with $\lambda = \lambda_\star$ concludes the proof.

\subsection{Linear convergence of OC-Halpern}
\label{subsec:oc-halpern-linear-conv}

\begin{lemma}
\label{lem:oc-halpern-dist-to-sol-rate}
Let $\opT \colon \reals^n \to \reals^n$ be a $\gamma^{-1}$-contractive operator, and let $w_k, w_{k+1/2}$ ($k=0,1,\dots$) be the OC-Halpern iterates, defined by
\begin{align*}
    w_{k+1/2} & = \beta_k w_0 + (1-\beta_k) w_k \\
    w_{k+1}   & = \opT (w_{k+1/2}),
\end{align*}
where $\beta_k = \left(\sum_{j=0}^k \gamma^{2j}\right)^{-1}$.
Then, for $k=0,1,\dots$,
\begin{align}
\label{eqn:oc_halpern_dist_to_sol_wkhalf}
    \|w_{k+1/2} - w_\star\|^2 & \le \left( \frac{\sum_{j=0}^k \gamma^j}{\sum_{j=0}^k \gamma^{2j}} \right)^2 \|w_0 - w_\star\|^2 
\end{align}
and
\begin{align}
\label{eqn:oc_halpern_dist_to_sol_wkp1}
    \|w_{k+1} - w_\star\|^2 & \le \frac{1}{\gamma^2} \left( \frac{\sum_{j=0}^k \gamma^j}{\sum_{j=0}^k \gamma^{2j}} \right)^2 \|w_0 - w_\star\|^2 .
\end{align}
\end{lemma}

We present the (simple but new) proof of this result in \Cref{sec:scsc_minimax_proofs}. For now, observe that
\begin{align*}
    \left( \frac{\sum_{j=0}^k \gamma^j}{\sum_{j=0}^k \gamma^{2j}} \right)^2 = \left( \frac{\frac{\gamma^{k+1} - 1}{\gamma - 1}}{\frac{\gamma^{2k+2} - 1}{\gamma^2 - 1}} \right)^2 = (1+\gamma)^2 \left( \frac{\gamma^{k+1} - 1}{\gamma^{2k+2} - 1} \right)^2 < \frac{(1+\gamma)^2}{\gamma^{2k+2}},
\end{align*}
so \cref{lem:oc-halpern-dist-to-sol-rate} implies the asymptotic rates $\|w_{k+1/2} - w_\star\|^2 = \cO \left( \gamma^{-2k} \|w_0 - w_\star\|^2 \right)$ and $\|w_{k+1} - w_\star\|^2 = \cO \left( \gamma^{-2k} \|w_0 - w_\star\|^2 \right)$.

\subsection{Proof of \texorpdfstring{\cref{thm:sm-eag-fast-point-convergence}}{Theorem 4.3}}
Let $\{w_k\}_{k=0,1,\dots}$ be the iterates of OC-Halpern with respect to the resolvent operator $\JB$ with $\beta_k = \frac{1}{\sum_{j=0}^k (1+2\alpha\mu)^j}$ (viewing $\JB$ as a $(1+2\alpha\mu)^{-1/2}$-contractive operator) and $w_0=z_0$.
Then
\begin{align*}
\|z_k - z_\star\|^2 = \|(w_k - z_\star) + (z_k - w_k)\|^2 \le 2\|w_k - z_\star\|^2 + 2\|z_k - w_k\|^2.
\end{align*}
The first term and the second term can be respectively bounded using \cref{lem:oc-halpern-dist-to-sol-rate} and \cref{thm:sm-EAG_OC-Halpern_MP}.
Applying \cref{lem:oc-halpern-dist-to-sol-rate} with $\gamma = (1+2\alpha\mu)^{1/2}$ yields the bound
\begin{align*}
    \|w_k - z_\star\|^2 & \le \frac{1}{\gamma^2} \left( \frac{\sum_{j=0}^{k-1} \gamma^j}{\sum_{j=0}^{k-1} \gamma^{2j}} \right)^2 \|z_0 - z_\star\|^2 
    \le \frac{(1+\gamma)^2}{\gamma^{2k+2}} \|z_0 - z_\star\|^2 .
\end{align*}
Substituting $\gamma = (1+2\alpha\mu)^{1/2}$ into the above bound gives the first term of \eqref{eqn:smeag-dist-to-sol-rate}.
Applying the bound on $\|z_k - w_k\|^2$ from \cref{thm:sm-EAG_OC-Halpern_MP} gives the second term of \eqref{eqn:smeag-dist-to-sol-rate}.
Combining these, the proof is complete.

\section{Near-optimal prox-grad-type accelerated minimax algorithm}
\label{sec:composite_minimax}

Consider the setup with minimax objective of the form $\lagrange(x,y) = \lagrange_p (x,y) + \lagrange_s (x,y)$, where $\lagrange_s$ is $L$-smooth.
Informally assume that $\lagrange_p$ is nonsmooth and proximable (we clarify this notion below).
We reformulate the minimax optimization problem into the monotone inclusion problem
\begin{align}
\label{eqn:operator_splitting}
    \begin{array}{ll}
        \underset{z \in \reals^{d}}{\textrm{find}} & 0 \in (\ssd\lagrange_p + \sop\lagrange_s) (z),
    \end{array}
\end{align}
where $\ssd$ is the analog of $\sop$ for nonsmooth convex-concave functions (precisely defined in \cref{subsec:composite_minimax_preliminaries}).
We design the \emph{\textbf{A}nchored \textbf{P}roximal \textbf{G}radient\textsuperscript{*} (APG\textsuperscript{*})}: 
\begin{align}
\label{eqn:APG*}
    \begin{aligned}
        z_k       & = \textbf{SM-EAG\texttt{+}} \left( \opI + \alpha\opB - \xi_k, \xi_k, \epsilon_k \right) \\
        \xi_{k+1} & = \beta_k z_0 + (1-\beta_k) \left( \JA(z_k - \alpha\opB z_k) + \alpha\opB z_k \right) ,
    \end{aligned}
    \tag{APG\textsuperscript{*}}
\end{align}
where $\opA = \ssd\lagrange_p$, $\opB = \sop\lagrange_s$, $\xi_0 \in \reals^d$ is the initial point, and $\textbf{SM-EAG\texttt{+}} \left( \opI + \alpha\opB - \xi_k, \xi_k, \epsilon_k \right)$ denotes the execution of \ref{eqn:SM-EAG+} with the $1$-strongly monotone $(1+\alpha L)$-smooth operator $z \mapsto (\opI + \alpha\opB)(z) - \xi_k$ with $\xi_k$ as the initial point using the step-size $\frac{\sqrt{(1+\alpha L)^2 + 1} + 1}{(1+\alpha L)^2}$ until reaching  the tolerance $\|(\opI + \alpha\opB)(\cdot) - \xi_k\| \le \epsilon_k$.

\ref{eqn:APG*} is a proximal-gradient-type algorithm; each iteration, \ref{eqn:APG*} uses a single proximal operation $\opJ_{\alpha\ssd\lagrange_p}$ and 
multiple (logarithmically many) gradient operations $\sop\lagrange_s$.
(The superscript * in the name \ref{eqn:APG*} indicates the multiple gradient evaluations per iteration.)
Thus, \ref{eqn:APG*} is useful when $\lagrange_p$ is \emph{proximable} in the sense that $\opJ_{\alpha\ssd\lagrange_p}$ can be computed efficiently.
\ref{eqn:APG*} is designed to approximate the accelerated proximal algorithm \ref{eqn:DRS-Halpern_alternate_form}, described soon, in the MP sense and thereby achieves a near-optimal accelerated rate.

\subsection{Preliminaries: Subdifferential operators and forward-backward residual}
\label{subsec:composite_minimax_preliminaries}

If $\lagrange_p\colon \reals^n \times \reals^m \to \reals$ is a finite-valued, not necessarily differentiable convex-concave function, then its \emph{(saddle) subdifferential} operator $\ssd \lagrange_p\colon \reals^n \times \reals^m \rightrightarrows \reals^n \times \reals^m$ defined by
\begin{align*}
    \ssd \lagrange_p (x,y) = \{(v,w)\in \reals^n \times \reals^m \,|\, v \in \partial_x \lagrange_p(x,y), w \in \partial_y (-\lagrange_p)(x,y) \}
\end{align*}
is maximal monotone \cite[Corollary~37.5.2]{Rockafellar1970_convex}, where $\partial_x$ and $\partial_y$ respectively denote the convex subdifferential with respect to the $x$- and $y$-variables.
However, in order to encompass a broader range of problems, such as constrained minimax optimization problems, it is necessary to consider $\lagrange_p$ taking values in $\reals\cup\{\pm\infty\}$.
In this case, $\partial\lagrange_p$ is maximal monotone if $\lagrange_p$ is closed and proper \cite[Corollary~37.5.2]{Rockafellar1970_convex}, and if that holds, $z_\star$ is a solution of~\eqref{eqn:minimax_problem} if and only if $0 \in \ssd\lagrange(z_\star)$.
(We refer the readers to Appendix~\ref{sec:regularity} for a more detailed discussion on properness.)
In our setup with $\lagrange = \lagrange_p + \lagrange_s$, we assume closedness and properness of $\lagrange_p$ and rewrite~\eqref{eqn:minimax_problem} into~\eqref{eqn:operator_splitting}, using the fact that $\ssd = \sop$ for differentiable convex-concave functions.

Of particular interest are problems with $\lagrange(x,y) = f(x) +  \langle Ax, y \rangle - g(y)$, with some $A\in\reals^{m\times n}$, 
or, more generally, $\lagrange(x,y) = f(x) + \lagrange_s(x,y) - g(y)$, with some $L$-smooth convex-concave $\lagrange_s$.
In this case, $\ssd \lagrange_p(x,y) = (\partial f(x), \partial g(y))$, and $\opJ_{\alpha \ssd\lagrange_p}(x,y) = (\prox_{\alpha f}(x), \prox_{\alpha g}(y))$, where
\begin{align*}
    \prox_{\alpha f}(x) = \underset{x'\in\reals^n}{\argmin} \left\{ \alpha f(x') + \frac{1}{2}\|x'-x\|^2 \right\}, \,\,\, \prox_{\alpha g}(y) = \underset{y'\in\reals^m}{\argmin} \left\{ \alpha g(y') + \frac{1}{2}\|y'-y\|^2 \right\} 
\end{align*}
are the proximal operators.
When $f(x) = \delta_\cX(x) = \begin{cases} 0 & x\in \cX \\ +\infty & x\notin \cX \end{cases}$, the proximal operator becomes the projection operator: $\prox_{\alpha\delta_\cX}(x) = \proj_{\cX}(x)$ for any $\alpha > 0$.
Therefore, the constrained problem~\eqref{eqn:minimax_problem} with $\cX \subset \reals^n$, $\cY \subset \reals^m$ is recovered with $\lagrange_p (x,y) = \delta_\cX (x) - \delta_\cY (y)$.

Note that $z_\star$ is a solution of~\eqref{eqn:operator_splitting} if and only if $z_\star = \JA(z_\star - \alpha\opB z_\star)$.
Therefore, the norm of the \emph{forward-backward residual} $\opG_\alpha (z) = \frac{1}{\alpha} (z - \JA(z - \alpha\opB z))$ serves as a measure for the convergence of algorithms solving~\eqref{eqn:operator_splitting}, and it generalizes the operator norm and the fixed-point residual norm.

\subsection{Preliminaries: Accelerated proximal splitting algorithm}
\label{subsec:composite_minimax_prior_algorithm}

Consider 
\begin{align}
\begin{aligned}
    w_k & = \JB (u_k) \\
    u_{k+1} & = \beta_k u_0 + (1-\beta_k) (\JA(w_k - \alpha\opB w_k) + \alpha \opB w_k)
\end{aligned}
\label{eqn:DRS-Halpern_alternate_form}
\tag{OHM-DRS}
\end{align}
with $\beta_k = \frac{1}{k+2}$.
As we explain below, this is the application of OHM to the Douglas--Rachford splitting (DRS) operator \cite{DouglasRachford1956_numerical, LionsMercier1979_splitting} defined by $\opT_{\textrm{DRS}} = \opI - \JB + \JA \circ (2\JB - \opI)$.
Here we use the alternative representation of OHM $u_{k+1} = \beta_k u_0 + (1-\beta_k) \opT u_k$ with $\beta_k = \frac{1}{k+2}$, which is equivalent to~\eqref{eqn:OHM} under the identification $u_k = w_{k+1/2}$.
Given $u_k$, letting $w_k = \JB(u_k)$ and $v_k = \JA\circ(2\JB - \opI)(u_k) = \JA(2w_k - u_k)$, we have $\TDRS(u_k) = u_k - w_k + v_k$.
Therefore, OHM with $\opT=\TDRS$ is written as
\begin{align}
    w_k     & = \JB (u_k) \nonumber \\
    v_k     & = \JA (2w_k - u_k) \label{eqn:DRS-Halpern_original_form} \\
    u_{k+1} & = \beta_k u_0 + (1-\beta_k) ( u_k + v_k - w_k ) \nonumber
\end{align}
where $\beta_k = \frac{1}{k+2}$.
But $2w_k - u_k = 2w_k - (w_k + \alpha\opB w_k) = w_k - \alpha\opB w_k$, so $v_k = \JA(w_k-\alpha\opB w_k)$.
By eliminating the second line and replacing $u_k$ by $w_k + \alpha\opB w_k$ in the third line of~\eqref{eqn:DRS-Halpern_original_form}, we obtain \ref{eqn:DRS-Halpern_alternate_form}.

Since $\opT_{\textrm{DRS}}$ is nonexpansive \cite{LionsMercier1979_splitting}, \ref{eqn:DRS-Halpern_alternate_form} exhibits the accelerated rate\\
$\|u_k - \opT_{\textrm{DRS}}(u_k)\|^2 \le 4\|u_0-u_\star\|^2/(k+1)^2$ if $u_\star \in \fix\TDRS$.
Additionally, this rate directly translates into the accelerated rate with respect to $\opG_\alpha (w_k)$, as
\begin{align*}
    u_k - \TDRS(u_k) = w_k - v_k  = w_k - \JA(w_k - \alpha\opB w_k) = \alpha \opG_\alpha(w_k) .
\end{align*}
\ref{eqn:DRS-Halpern_alternate_form} will serve as a conceptual algorithm (not implemented) that our \ref{eqn:APG*} (implementable) is designed to $\mathcal{O}(1/k^2)$-MP approximate, so that \ref{eqn:APG*} inherits the accelerated rate of \ref{eqn:DRS-Halpern_alternate_form}.
(\ref{eqn:DRS-Halpern_alternate_form} is not prox-grad-type, while \ref{eqn:APG*} is.)

\subsection{MP property and convergence of APG\texorpdfstring{\textsuperscript{*}}{*}}
\label{subsec:composite_minimax_APG*_explanation}
We now describe the design of \ref{eqn:APG*}. 
The forms of \ref{eqn:APG*} and \ref{eqn:DRS-Halpern_alternate_form} bear a clear resemblance. The second lines are identical.
In the first lines, the resolvent computation $\JB$ of \ref{eqn:DRS-Halpern_alternate_form} is replaced in \ref{eqn:APG*} by an inner loop with \ref{eqn:SM-EAG+}.
Note, $z = \JB(\xi_k)$ if and only if it solves
\begin{align}
\label{eqn:inner_loop}
    \begin{array}{ll}
        \underset{z \in \reals^{d}}{\textrm{find}} & 0 = z + \alpha \opB z - \xi_k = (\opI + \alpha\opB - \xi_k)(z),
    \end{array}
\end{align}
and $\opI + \alpha\opB - \xi_k$ is a $1$-strongly monotone, $(1+\alpha L)$-Lipschitz operator.
Therefore, \eqref{eqn:inner_loop} is an instance of the problem class of \Cref{sec:scsc_minimax}, and \ref{eqn:SM-EAG+} can efficiently solve it to accuracy $\|(\opI + \alpha\opB)(z) - \xi_k\| \le \epsilon_k$ in $\mathcal{O}(\log (1/\epsilon_k))$ iterations.
Therefore the $z_k$-iterate of \ref{eqn:APG*} approximates the $w_k = \JB(u_k)$-iterate of \ref{eqn:DRS-Halpern_alternate_form}.

Throughout the following theorems, we assume that $\zer(\opA + \opB) \ne \emptyset$.
As $z_\star \in \zer(\opA + \opB)$ if and only if there exists $u_\star \in \fix(\TDRS)$ such that $z_\star = \JB(u_\star)$ \cite{LionsMercier1979_splitting}, this implies $\fix(\TDRS) \ne \emptyset$.
We defer the proofs to \Cref{sec:composite_minimax_proofs}.

\begin{theorem}[\ref{eqn:APG*}$\approx$\ref{eqn:DRS-Halpern_alternate_form}]
\label{thm:APG_MP}
With $\alpha\in\left( 0, \frac{1}{L} \right)$, $\beta_k = \frac{1}{k+2}$, and $\epsilon_k = \frac{1+L^{-1}\|\opB \xi_0\|}{(k+1)^2 (k+2)}$, \ref{eqn:APG*} is $\mathcal{O}(1/k^2)$-MP to~\ref{eqn:DRS-Halpern_alternate_form}; specifically, if $\xi_k ,z_k$ are the \ref{eqn:APG*} iterates and $u_k,w_k$ are the iterates of~\ref{eqn:DRS-Halpern_alternate_form}, and if $\xi_0 = u_0$, then 
\begin{align*}
    \max \left\{ \|\xi_k - u_k\|^2, \|z_k - w_k\|^2 \right\} \le \frac{C(\xi_0)^2}{L^2(k+1)^2} ,
    \qquad
    \text{for }k=0,1,\dots,
\end{align*}
where $C(\xi_0) = L(\|\xi_0-\xi_\star\| + 1) + \|\opB \xi_\star\|$ and $\xi_\star = \proj_{\fix(\opT_{\textup{DRS}})}(\xi_0)$.
\end{theorem}

\begin{corollary}[Fast rate of \ref{eqn:APG*}]
\label{cor:APG_convergence}
With the conditions and notations of \cref{thm:APG_MP}, $\xi_k \to \xi_\star$ and $z_k \to \JB(\xi_\star) \in \zer(\opA+\opB)$ as $k\to\infty$.
Additionally,
\begin{align*}
    \|\opG_\alpha (z_k)\|^2 \le \frac{(3+\alpha L)^2 C(\xi_0)^2}{\alpha^2 L^2 (k+1)^2}
    ,
    \qquad
    \text{for }k=0,1,\dots.
\end{align*}
\end{corollary}

\begin{corollary}[Oracle complexity of \ref{eqn:APG*}]
\label{cor:APG_complexity}
With the conditions and notations of \cref{thm:APG_MP}, the $k$-th inner loop of \ref{eqn:SM-EAG+} in \ref{eqn:APG*} requires $\mathcal{O} \left( \log \left( k\left( \|\xi_0-\xi_\star\|+1 \right)\right)\right)$ oracle calls of $\opB$ for $k=0,1,\dots$.
This implies that the total computational cost needed to achieve $\|\opG_\alpha (z_k)\| \le \epsilon$ is 
\begin{align*}
    \mathcal{O}\left( \frac{C(\xi_0)}{\epsilon} \left( C_\opA + C_\opB \log \frac{C(\xi_0)}{\epsilon} \right) \right) ,
\end{align*}
where $C_\opA, C_\opB$ are upper bounds on the respective costs of evaluating $\JA$ and $\opB$.
\end{corollary}

We clarify some key differences between \ref{eqn:APG*} and the prior work of Diakonikolas \cite[Algorithm~3]{Diakonikolas2020_halpern}.
First, \ref{eqn:APG*} is a more general proximal-gradient-type algorithm, while the prior algorithm of Diakonikolas is a projected-gradient-type algorithm.
Second, \ref{eqn:APG*} is an anytime algorithm, as it requires no a priori specification of the terminal accuracy or the total iteration count, while the prior algorithm is not an anytime algorithm, as it requires the terminal accuracy $\epsilon$ to be pre-specified.
Third, \ref{eqn:APG*} has $\mathcal{O}\left(\frac{C_\opA}{\epsilon} + \frac{C_\opB}{\epsilon}\log \frac{1}{\epsilon}\right)$ total cost (hiding dependency on $\xi_0, \xi_\star$ and $L$), while the prior algorithm has $\mathcal{O}\left(\frac{C_\opA+C_\opB}{\epsilon}\log \frac{1}{\epsilon}\right)$ total cost, as its inner loops require projections (corresponding to $\JA$). 
Therefore, \ref{eqn:APG*} offers a speedup for constrained minimax problems when the projection represents the dominant cost, i.e., when $C_\opA \gg C_\opB$.

The initial point $\xi_k$ for the inner loop $\textbf{SM-EAG\texttt{+}} \left( \opI + \alpha\opB - \xi_k, \xi_k, \epsilon_k \right) $ of \ref{eqn:APG*} is a crude choice. However, while a more careful warm starting may improve the iteration complexity by a constant factor, we do not believe it alone can eliminate the $\mathcal{O}(\log(1/\epsilon))$-factor. In our subsequent analysis of~\ref{eqn:APG*}, we induce the MP property by solving the $k$-th inner loop up to accuracy $\epsilon_k$ satisfying $\sum_{k=0}^\infty k\epsilon_k < \infty$. Therefore, unless a mechanism for finding a warm-starting point $z$ satisfying $\|(\opI + \alpha\opB)(z) - \xi_k\| = \cO(\epsilon_k) = o(1/k)$ is proposed (which seems nontrivial), the $\mathcal{O}(\log(1/\epsilon))$-factor will persist. Eliminating the logarithmic factor probably requires a new additional insight, and we leave it as a future work.

\section{Numerical experiments}
\label{sec:numerical-simulation}

In this section, we provide numerical experiments.
First, we compare \smeag~with EG and OG on the following bilinearly coupled minimax optimization problem:
\begin{align}
\label{eqn:smeag-experiment-objective}
    \underset{x \in \reals^d}{\textrm{minimize}} \,\, \underset{y \in \reals^d}{\textrm{maximize}} \,\,\,\, \lagrange(x, y) = \frac{\mu}{2} \|x\|^2 + x^\intercal A y - \frac{\mu}{2} \|y\|^2 .
\end{align}
For completeness, we provide the precise forms of EG and OG in \Cref{sec:algorithm_specifications}.
We generate $A$ as a random matrix whose entries are i.i.d.\ Gaussians $\cN(0, \sigma^2)$.
We choose $d=50$, $\sigma=10^3$, and choose $\mu$ so that the condition number $\frac{L}{\mu}$ becomes $10^5$.
EG and OG converge slowly with respect to both measures of convergence: $\|\sop\lagrange(x_k, y_k)\|^2$ and $\|(x_k, y_k) - (x_\star, y_\star)\|^2$ (\cref{fig:sm-eag-experiment}), while \smeag~quickly reduces the both measures, despite the problem being ill-behaved with a large condition number.

\begin{figure*}[ht]
\captionsetup[subfigure]{justification=centering}
    \centering
    \begin{subfigure}{0.48\linewidth}
    \centering
    \includegraphics[scale=0.41]{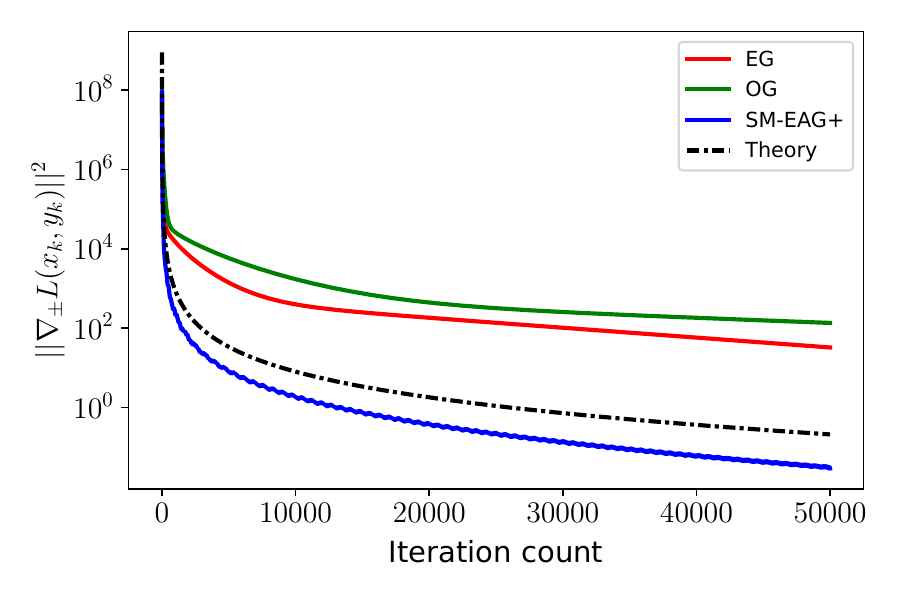}
    \caption{Plot of $\|\nabla_\pm \lagrange(x_k, y_k)\|^2$}
    \label{subfig:sm-eag-grad-norm}
    \end{subfigure}
    \hspace{.2cm}
    \begin{subfigure}{0.48\linewidth}
    \centering
    \includegraphics[scale=0.41]{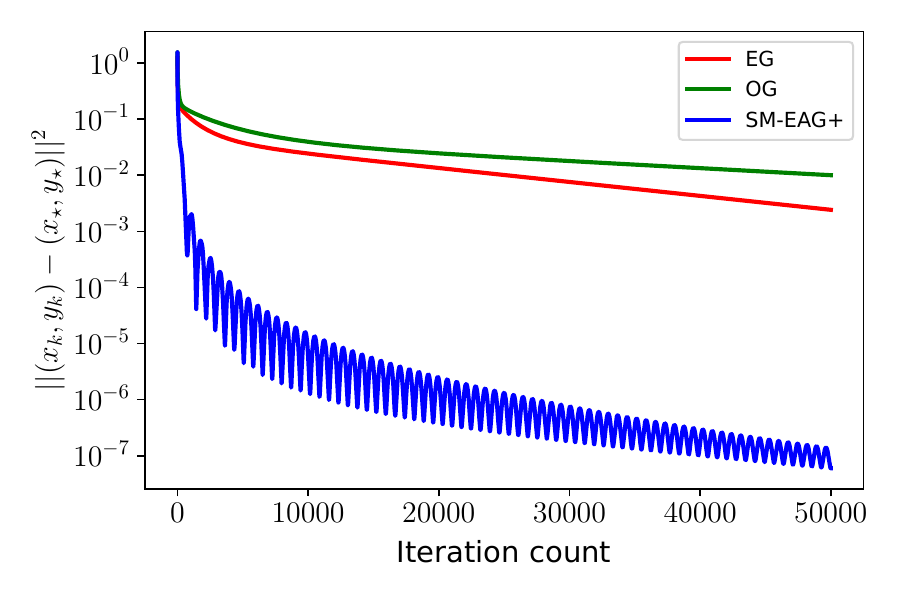}
    \caption{Plot of $\|(x_k,y_k) - (x_\star,y_\star)\|^2$}
    \label{subfig:sm-eag-dist-to-sol}
    \end{subfigure}
    \caption[]{
    Comparison of EG, OG and \smeag~on the problem~\eqref{eqn:smeag-experiment-objective}. The condition number $\frac{L}{\mu}$ of the problem is $10^5$. We use the largest possible step-size for \smeag, and tuned step-sizes for EG and OG (for best performance with respect to the gradient norm). In the left figure, the dashed black line labeled as ``Theory'' indicates the theoretical upper bound on $\|\opB z_k\|^2$ from \cref{thm:sm-EAG}.
    }
    \label{fig:sm-eag-experiment}
\end{figure*}

\begin{figure*}[ht]
\captionsetup[subfigure]{justification=centering}
    \centering
    \begin{subfigure}{0.48\linewidth}
    \centering
    \includegraphics[scale=0.41]{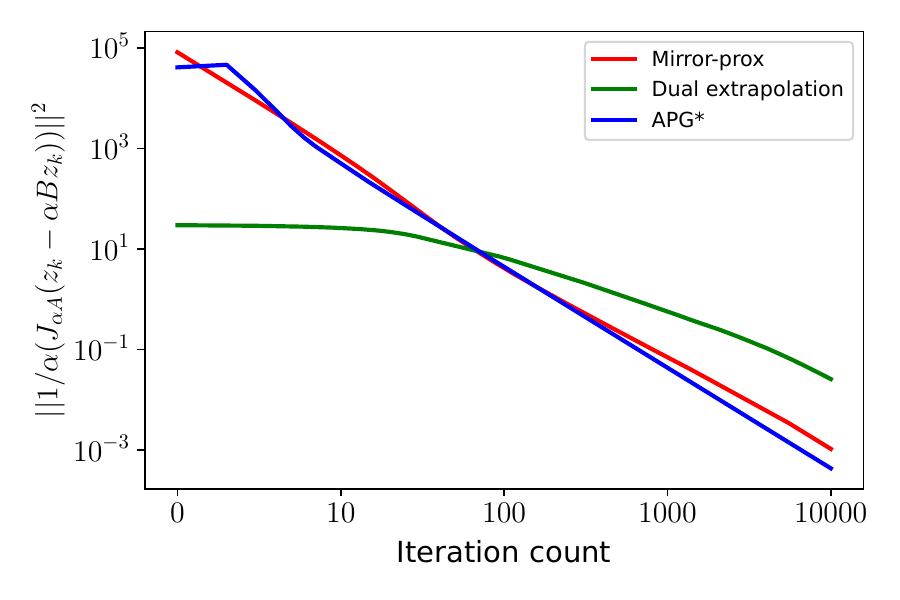}
    \end{subfigure}
    \caption[]{
    Comparison of mirror-prox, dual extrapolation and \apg~on the problem~\cref{eqn:apg-experiment-objective}, in terms of the squared forward-backward residual norm.
    }
    \label{fig:apg-experiment}
\end{figure*}

Next, we compare \apg~with mirror-prox~\cite{Nemirovski2004_proxmethod} and dual extrapolation (DE)~\cite{Nesterov2007_dual} algorithms (designed for constrained smooth variational inequalities) on the following bilinearly coupled minimax optimization problem (matrix game), constrained on probability simplices:
\begin{align}
\label{eqn:apg-experiment-objective}
    \underset{x \in \Delta^n}{\textrm{minimize}} \,\, \underset{y \in \Delta^m}{\textrm{maximize}} \,\,\,\, \lagrange(x, y) = x^\intercal A y .
\end{align}
For completeness, we provide the precise forms of mirror-prox and DE in \Cref{sec:algorithm_specifications}.
We generate $A \in \reals^{n\times m}$ as a random matrix whose entries are i.i.d.\ Gaussians $\cN(0, 1)$.
We choose $n=2\times 10^3$ and $m=10^3$, and we have $L\approx 75$.
For mirror-prox and \apg, the initial point is chosen randomly with i.i.d.\ Gaussian $\cN(0, \sigma^2)$ entries with $\sigma=0.05$.
We use the step-size $\alpha = \frac{1}{L}$ for \apg~and dual extrapolation, and $\alpha = \frac{1}{\sqrt{2}L}$ for mirror-prox (as recommended by the original papers).
We observe fast $\cO(1/k^2)$ convergence of \apg, and dual extrapolation is the slowest.
For DE, there is an initial scale difference because, unlike the other two algorithms that use randomly initialized $(x_0, y_0)$, DE is initialized with $s_{-1} = 0$ in the dual space (so there is no explicit primal initial point), and the primal iterates $(x_k, y_k)$ are generated based on the dual iterates $s_k$.
The initial forward-backward residual norms from \apg~and mirror-prox are different as the values of $\alpha$ are different (although $(x_0, y_0)$ is the same).

\section{Conclusion}
In this paper, we identified the novel merging path (MP) property among accelerated minimax algorithms, EAG \cite{YoonRyu2021_accelerated}, FEG \cite{LeeKim2021_fast}, and APS \cite{Tran-DinhLuo2021_halperntype}, and the accelerated proximal algorithm OHM \cite{Halpern1967_fixed, Lieder2021_convergence, Kim2021_accelerated}. We then used this MP property to design new state-of-the-art algorithms.

We quickly mention that our results straightforwardly extend to the setup where the underlying space $\reals^d$ is replaced with an infinite-dimensional Hilbert space, since our proofs do not rely on the finite-dimensionality of $\reals^d$. In particular, the point convergence results of \cref{cor:point-conv} and \cref{cor:APG_convergence} hold as strong convergence.

The MP property provides the insight that the four different accelerated algorithms, EAG, FEG, APS, and OHM, actually reduce to one single acceleration mechanism. However, elucidating the fundamental principle behind this single anchoring-based acceleration mechanism still remains an open problem, and further investigating the phenomenon from multiple different viewpoints, as was done for Nesterov's acceleration, would be an interesting direction of future work.

The presence or applicability of the MP property in other setups is another interesting direction of future work. One question is whether other prior optimization algorithms exhibit the MP property. In our analysis, we found that anchoring induces the MP property, so algorithms using similar regularizations, such as the recursive regularization \cite{Allen-Zhu2018_how} or the Katyusha momentum \cite{Allen-Zhu2017_katyusha}, may be approximating some other conceptual algorithm. Another question is whether the MP property can be used to design new algorithms in other setups. The strategy of designing efficient algorithms by approximating a conceptual algorithm, in the MP sense, may be widely applicable in setups beyond non-stochastic minimax optimization. 

\section*{Acknowledgments}
This work was supported by the Samsung Science and Technology Foundation (project SSTF-BA2101-02), the National Research Foundation of Korea (NRF) grant funded by the Korean Government (RS-2024-00421203, RS-2024-00406127).
We thank Jongmin Lee and Soheun Yi for providing valuable feedback.

\newpage

\bibliography{ref}

\newpage

\appendix

\section{Algorithm specifications}
\label{sec:algorithm_specifications}

\subsection{Classical minimax optimization algorithms}

\emph{Simultaneous gradient descent-ascent} for smooth minimax optimization is defined by
\begin{align*}
    & x_{k+1} = x_k - \alpha \nabla_x \lagrange(x_k, y_k) \\ 
    & y_{k+1} = y_k + \alpha \nabla_y \lagrange(x_k, y_k)
\end{align*}
where $\alpha > 0$ is the step-size.
Writing $z_k = (x_k, y_k)$ and using the $\sop$ notation, we can more concisely rewrite this as
\begin{align*}
    z_{k+1} = z_k - \alpha \sop \lagrange (z_k).
\end{align*}
The \emph{extragradient (EG)} algorithm is defined by
\begin{align*}
    z_{k+1/2} & = z_k - \alpha \sop \lagrange (z_k), \\
    z_{k+1} & = z_k - \alpha \sop \lagrange (z_{k+1/2}).
\end{align*}
The \emph{optimistic gradient descent (OG)} algorithm is defined by
\begin{align*}
    z_{k+1} = z_k - \alpha \sop \lagrange (z_k) - \alpha \left( \sop\lagrange(z_k) - \sop\lagrange(z_{k-1}) \right),
\end{align*}
where $z_0=z_{-1}$ are the starting points.

\subsection{Nesterov's AGM and its momentum parameters}
\label{subsec:AGM_and_FISTA}

We specify the versions of Nesterov's AGM used in \cref{subfig:FISTA_divergent_paths}.
Given $f\colon \reals^n \to \reals$ which is convex and $L$-smooth, AGM takes the form
\begin{align*}
    x_{k+1} & = y_k - \alpha \nabla f(y_k) \\
    y_{k+1} & = x_{k+1} + \frac{t_k - 1}{t_{k+1}} (x_{k+1} - x_k) 
\end{align*}
for $k=0,1,\dots$.
It converges at the accelerated rate $f(x_k) - f_\star \le \mathcal{O}(\|x_0-x_\star\|^2/k^2)$ provided that $\alpha \le \frac{1}{L}$ and $t_{k+1}^2 - t_{k+1} - t_k^2 \le 0$ \cite{Nesterov1983_method, ChambolleDossal2015_convergence}.
In particular, we take $t_k = \frac{k+a-1}{a}$ (which was used in \cite{ChambolleDossal2015_convergence}) and refer to $a$ as the momentum parameter in \cref{subfig:FISTA_divergent_paths}.

\subsection{Algorithms for constrained minimax problems}
\label{subsec:mirror-prox-and-dual-extrapolation}

We specify the algorithms compared with \apg~in \Cref{sec:numerical-simulation}.
We denote by $\cZ = \cX \times \cY$ the constraint set in \cref{eqn:minimax_problem}.
The \emph{mirror-prox} \cite{Nemirovski2004_proxmethod} algorithm is defined by:
\begin{align*}
    z_{k+1/2} & = \proj_{\cZ} \left( z_k - \alpha \sop\lagrange(z_k) \right) \\
    z_{k+1} & = \proj_{\cZ} \left( z_k - \alpha \sop\lagrange(z_{k+1/2}) \right)
\end{align*}
and $\overline{z}_k := \frac{1}{k+1} \sum_{j=0}^k z_j$.
The final output of the algorithm is the averaged iterate $\overline{z}_k$.
The \emph{dual extrapolation} (DE) \cite{Nesterov2007_dual} algorithm is defined by:
\begin{align*}
    u_k & = \proj_{\cZ} \left( z_c + \alpha s_{k-1} \right) \\
    z_k & = \proj_{\cZ} \left( u_k - \alpha \sop\lagrange (u_k) \right) \\
    s_k & = s_{k-1} - \sop\lagrange (z_k) 
\end{align*}
and $\overline{z}_k := \frac{1}{k+1} \sum_{j=0}^k z_j$, where $s_{-1} = 0$ and $z_c \in \cZ$ is a center point, fixed throughout the algorithm iterations.
The final output of the algorithm is the averaged iterate $\overline{z}_k$.

\section{Omitted proofs for \texorpdfstring{\Cref{sec:smooth_minimax}}{Section 3}}
\label{sec:smooth_minimax_proofs}

\subsection{Proof of \texorpdfstring{\cref{lem:FEG_summability}}{Lemma 3.2}}
\label{sec:FEG_summability_lemma_proof}

Let $V_k = \frac{\alpha k^2}{2}\|\opB z_k\|^2 + k \langle \opB z_k, z_k - z_0 \rangle + \frac{1}{2\alpha} \|z_0 - z_\star\|^2$.
Then $V_k \ge \frac{\alpha k^2}{2}\|\opB z_k\|^2 + k\langle \opB z_k, z_\star - z_0 \rangle + \frac{1}{2\alpha} \|z_0 - z_\star\|^2 \ge 0$ for $k=0,1,\dots$, where the first inequality follows from monotonicity of $\opB$ and the second one is Young's inequality.
Next, we show
\begin{align}
\label{eqn:FEG_summability_key_inequality}
    V_k - V_{k+1} \ge \frac{\alpha (1-\alpha^2 L^2)}{2} \|k \opB z_k - (k+1) \opB z_{k+1/2}\|^2 , \qquad\text{for } k=0,1,\dots.
\end{align}
Once this is established, the conclusion follows by taking $N\to\infty$ in
\begin{align*}
    \frac{\alpha (1-\alpha^2 L^2)}{2} \sum_{k=0}^N \|k \opB z_k - (k+1)\opB z_{k+1/2}\|^2 \le V_0 - V_{N+1} \le V_0 = \frac{1}{2\alpha}\|z_0 - z_\star\|^2 .
\end{align*}

First, consider the case $k=0$, where $z_{1/2}=z_0$ and $z_1 = z_0 - \alpha\opB z_{1/2} = z_0 - \alpha\opB z_0$, $V_0 = \frac{1}{2\alpha}\|z_0-z_\star\|^2$, and
\begin{align*}
    V_1 & = \frac{\alpha}{2}\|\opB z_1\|^2 + \langle \opB z_1, z_1 - z_0 \rangle + \frac{1}{2\alpha}\|z_0-z_\star\|^2 \nonumber \\
    & = \frac{\alpha}{2}\|\opB z_1\|^2 - \alpha \langle \opB z_1, \opB z_0 \rangle + \frac{1}{2\alpha}\|z_0-z_\star\|^2 . 
\end{align*}
Because $\opB$ is $L$-Lipschitz,
\begin{align*}
    \|\opB z_1 - \opB z_0\|^2 \le L^2 \|z_1 - z_0\|^2 = \alpha^2 L^2 \|\opB z_0\|^2
\end{align*}
and by multiplying $\frac{\alpha}{2}$ on both sides and rearranging terms, we get
\begin{align*}
    \frac{\alpha}{2}\|\opB z_1\|^2 - \alpha\langle \opB z_1, \opB z_0 \rangle \le - \frac{\alpha}{2} (1-\alpha^2 L^2) \|\opB z_0\|^2 = - \frac{\alpha}{2} (1-\alpha^2 L^2) \|\opB z_{1/2}\|^2 .
\end{align*}
Adding $\frac{1}{2\alpha}\|z_0-z_\star\|^2$ to both sides, we obtain~\eqref{eqn:FEG_summability_key_inequality} for $k=0$.

Now suppose $k\ge 1$. For FEG, we have the identities
\begin{align*}
    z_{k+1} - z_0 & = (1-\beta_k) (z_k - z_0) - \alpha \opB z_{k+1/2} \\
    z_k - z_{k+1} & = \beta_k (z_k - z_0) + \alpha \opB z_{k+1/2} \\
    z_{k+1/2} - z_{k+1} & = \alpha \opB z_{k+1/2} - \alpha (1-\beta_k) \opB z_k 
\end{align*}
with $\beta_k = \frac{1}{k+1}$.
Therefore, 
using monotonicity of $\opB$, we obtain
\begin{align}
    & V_k - V_{k+1} \nonumber \\
    & \ge V_k - V_{k+1} - k(k+1) \langle z_k - z_{k+1} , \opB z_k - \opB z_{k+1} \rangle \nonumber \\
    & = \frac{\alpha k^2}{2} \|\opB z_k\|^2 + k \langle \opB z_k, z_k - z_0 \rangle \nonumber \\
    & \quad - \frac{\alpha (k+1)^2}{2} \|\opB z_{k+1}\|^2 - (k+1) \langle \opB z_{k+1}, (1-\beta_k) (z_k - z_0) - \alpha \opB z_{k+1/2} \rangle \nonumber \\
    & \quad - k(k+1) \langle  \beta_k (z_k - z_0) + \alpha \opB z_{k+1/2}, \opB z_k - \opB z_{k+1}  \rangle \nonumber \\
    & \begin{aligned}
        & = \frac{\alpha k^2}{2} \|\opB z_k\|^2 - \alpha k(k+1) \langle \opB z_k, \opB z_{k+1/2} \rangle \\
        & \quad + \alpha (k+1)^2 \langle \opB z_{k+1/2} , \opB z_{k+1} \rangle - \frac{\alpha (k+1)^2}{2} \|\opB z_{k+1}\|^2 .
    \end{aligned} \label{eqn:MP_FEG_Vk-Vkp1_lowerbound}
\end{align}
Now multiplying the factor $\frac{\alpha (k+1)^2}{2}$ to the inequality
\begin{align*}
    0 & \le L^2 \|z_{k+1/2} - z_{k+1}\|^2 - \|\opB z_{k+1/2} - \opB z_{k+1}\|^2 \\
    & = \alpha^2 L^2 \left\|\opB z_{k+1/2} - \frac{k}{k+1} \opB z_k \right\|^2 - \|\opB z_{k+1/2} - \opB z_{k+1}\|^2 .
\end{align*}
and subtracting from \eqref{eqn:MP_FEG_Vk-Vkp1_lowerbound} gives
\begin{align*}
    & V_k - V_{k+1} \\
    & \ge \frac{\alpha k^2}{2} \|\opB z_k\|^2 - \alpha k(k+1) \langle \opB z_k, \opB z_{k+1/2} \rangle + \alpha (k+1)^2 \langle \opB z_{k+1/2} , \opB z_{k+1} \rangle - \frac{\alpha (k+1)^2}{2} \|\opB z_{k+1}\|^2 \\
    & \quad - \frac{\alpha (k+1)^2}{2} \left( \alpha^2 L^2 \left\|\opB z_{k+1/2} - \frac{k}{k+1} \opB z_k \right\|^2 - \|\opB z_{k+1/2} - \opB z_{k+1}\|^2 \right) \\
    & = \frac{\alpha (1 - \alpha^2 L^2)}{2} \|k \opB z_k - (k+1) \opB z_{k+1/2}\|^2 ,
\end{align*}
which is the desired result.

\subsection{Remaining proof of Theorem~\ref{thm:EAG_FEG_Halpern_MP}}
\label{sec:EAG_APS_MP_proof}

We first consider the case of EAG.
As before, denote the Halpern iterates by $\{w_k\}_{k=0,1,\dots}$ and the EAG iterates by $\{z_k\}_{k=0,1,\dots}$.
Then in the exact same way as before, we have
\begin{align*}
    z_{k+1} - w_{k+1} & = \left( \beta_k z_0 + (1-\beta_k) z_k - \alpha \opB z_{k+1/2} \right) - \left( \beta_k w_0 + (1-\beta_k) w_k - \alpha \opB w_{k+1} \right) \\
    & = (1-\beta_k) (z_k - w_k) + \alpha \left( \opB w_{k+1} - \opB z_{k+1/2} \right) ,
\end{align*}
and
\begin{align}
\label{eqn:EAG_MP_identity}
    \begin{aligned}
    \left\| z_{k+1} - w_{k+1} \right\|^2 & = (1-\beta_k)^2 \left\|z_k - w_k\right\|^2 + 2 \left\langle \alpha (1-\beta_k) (z_k - w_k), \opB w_{k+1} - \opB z_{k+1/2} \right\rangle \\
    & \quad\quad + \alpha^2 \left\| \opB w_{k+1} - \opB z_{k+1/2} \right\|^2 .    
    \end{aligned}
\end{align}
The different part is
\begin{align*}
    z_{k+1/2} - w_{k+1} = (1-\beta_k) (z_k - w_k) - \alpha \opB z_k + \alpha \opB w_{k+1}
\end{align*}
so that inner product term in~\eqref{eqn:EAG_MP_identity} is rewritten as
\begin{align*}
    & 2 \left\langle \alpha (1-\beta_k) (z_k - w_k), \opB w_{k+1} - \opB z_{k+1/2} \right\rangle \\
    & = 2 \left\langle \alpha (z_{k+1/2} - w_{k+1}) + \alpha^2 \left( \opB z_k - \opB w_{k+1} \right), \opB w_{k+1} - \opB z_{k+1/2} \right\rangle \\
    & \le 2\alpha^2 \left\langle \opB z_k - \opB w_{k+1}, \opB w_{k+1} - \opB z_{k+1/2} \right\rangle .
\end{align*}
Combining this with~\eqref{eqn:EAG_MP_identity}, we obtain
\begin{align*}
    & \left\| z_{k+1} - w_{k+1} \right\|^2 \\
    & \le (1-\beta_k)^2 \left\|z_k - w_k\right\|^2 + 2\alpha^2 \left\langle \opB z_k - \opB w_{k+1}, \opB w_{k+1} - \opB z_{k+1/2} \right\rangle + \alpha^2 \left\| \opB w_{k+1} - \opB z_{k+1/2} \right\|^2 \\
    & = (1-\beta_k)^2 \left\|z_k - w_k\right\|^2 \underbrace{ -\, \alpha^2 \left\| \opB w_{k+1} \right\|^2  + 2\alpha^2 \langle \opB z_k, \opB w_{k+1} \rangle}_{\le \alpha^2 \|\opB z_k\|^2} - 2\alpha^2 \langle \opB z_k, \opB z_{k+1/2} \rangle + \alpha^2 \|\opB z_{k+1/2}\|^2
\end{align*}
Plugging in $\beta_k = \frac{1}{k+1}$ and multiplying both sides by $(k+1)^2$ gives
\begin{align*}
    (k+1)^2 \left\| z_{k+1} - w_{k+1} \right\|^2 \le & \, k^2 \left\| z_k - w_k \right\|^2 + \alpha^2 (k+1)^2 \|\opB z_k - \opB z_{k+1/2}\|^2 .
\end{align*}
Now, with the following Lemma~\ref{lem:EAG_summability}, the proof for EAG is done.

\begin{lemma}
\label{lem:EAG_summability}
Let $\opB\colon \reals^d \to \reals^d$ be monotone and $L$-Lipschitz and assume $z_\star \in\zer\opB$ exists.
Let $\{z_k\}_{k=0,1,\dots}$ be the iterates of EAG with $\beta_k = \frac{1}{k+1}$.
Then there exist constants $\eta \in (0,1)$ and $C>0$, not depending on $L$, such that for $\alpha \in \left(0, \frac{\eta}{L}\right)$,
\begin{align*}
    \sum_{k=0}^\infty (k+1)^2 \|\opB z_k - \opB z_{k+1/2}\|^2 \le \frac{C}{\alpha^2} \|z_0 - z_\star\|^2 < \infty .
\end{align*}
\end{lemma}

For the case of APS, proceeding with $\opB z_k, \opB z_{k+1/2}$ replaced by $\opB v_k, \opB v_{k+1}$ respectively, we obtain
\begin{align*}
    (k+1)^2 \left\| z_{k+1} - w_{k+1} \right\|^2 \le & \, k^2 \left\| z_k - w_k \right\|^2 + \alpha^2 (k+1)^2 \|\opB v_k - \opB v_{k+1}\|^2 
\end{align*}
and the proof is complete once we show:

\begin{lemma}
\label{lem:APS_summability}
Let $\opB\colon \reals^d \to \reals^d$ be monotone and $L$-Lipschitz and assume $z_\star \in\zer\opB$ exists.
Let $\{z_k\}_{k=0,1,\dots}$ be the iterates of APS with $\beta_k = \frac{1}{k+1}$.
Then there exist constants $\eta \in (0,1)$ and $C>0$, not depending on $L$, such that for $\alpha \in \left(0, \frac{\eta}{L}\right)$,
\begin{align*}
    \sum_{k=0}^\infty (k+1)^2 \|\opB v_k - \opB v_{k+1}\|^2 \le \frac{C}{\alpha^2} \|z_0 - z_\star\|^2 < \infty .
\end{align*}
\end{lemma}

The proofs of Lemmas~\ref{lem:EAG_summability} and \ref{lem:APS_summability} follow from arguments similar to that of Lemma~\ref{lem:FEG_summability}.
However, they turn out to be more lengthy, so we defer their proofs to Appendix~\ref{sec:summability_lemmas}.

\section{Omitted proofs for \texorpdfstring{\Cref{sec:scsc_minimax}}{Section 4}}
\label{sec:scsc_minimax_proofs}

\subsection{Proof of \texorpdfstring{\cref{lem:sm-EAG_Lyapunov}}{Lemma 4.4}}
When $k=0$, we have $z_{1/2} = z_0$, $z_1 = z_0 - \alpha \opB z_0$, $V_0 = \left( \frac{1}{2\alpha} + \mu\right) \|z_0 - z_\star\|^2$, and
\begin{align*}
    V_1 & = \frac{\alpha}{2}\|\opB z_1\|^2 + \langle \opB z_1 - \mu(z_1 - z_0), z_1 - z_0 \rangle + \left(\frac{1}{2\alpha} + \mu\right) \|z_0 - z_\star\|^2 \nonumber \\
    & = \frac{\alpha}{2}\|\opB z_1\|^2 - \alpha \langle \opB z_1, \opB z_0 \rangle - \alpha^2 \mu \|\opB z_0\|^2 + \left(\frac{1}{2\alpha} + \mu\right) \|z_0 - z_\star\|^2 . 
\end{align*}
Because $\opB$ is $L$-Lipschitz,
\begin{align*}
    0 & \le L^2 \|z_1 - z_0\|^2 - \|\opB z_1 - \opB z_0\|^2 = (\alpha^2 L^2 - 1) \|\opB z_0\|^2 + 2\langle \opB z_0, \opB z_1 \rangle - \|\opB z_1\|^2 .
\end{align*}
Therefore,
\begin{align*}
    V_0 - V_1 & = \alpha^2 \mu \|\opB z_0\|^2 + \alpha \langle \opB z_0, \opB z_1 \rangle - \frac{\alpha}{2} \|\opB z_1\|^2 \\
    & = \frac{\alpha}{2} (1+2\alpha\mu -\alpha^2 L^2) \|\opB z_0\|^2 + \frac{\alpha}{2} \left( (\alpha^2 L^2 - 1) \|\opB z_0\|^2 + 2\langle \opB z_0, \opB z_1 \rangle - \|\opB z_1\|^2 \right) \\
    & \ge \frac{\alpha}{2} (1+2\alpha\mu -\alpha^2 L^2) \|\opB z_0\|^2 \\
    & = \frac{\alpha}{2} (1+2\alpha\mu -\alpha^2 L^2) \|\eta_0 \opB z_0 - \opB z_{1/2}\|^2 
\end{align*}
where the last equality holds because $\eta_0 = 0$ and $z_{1/2} = z_0$.

Now suppose $k\ge 1$. For \ref{eqn:SM-EAG+}, we have the identities
\begin{align}
    z_{k+1} - z_0 & = (1-\beta_k) (z_k - z_0) - \alpha \opB z_{k+1/2} \label{eqn:sm_EAG+_zkp1_minus_z0} \\
    z_k - z_{k+1} & = \beta_k (z_k - z_0) + \alpha \opB z_{k+1/2} \label{eqn:SM-EAG+_zk_minus_zkp1} \\
    z_{k+1/2} - z_{k+1} & = \alpha \opB z_{k+1/2} - \eta_k \alpha \opB z_k . \label{eqn:sm_EAG+_zkhalf_minus_zkp1}
\end{align}
Note that $\beta_k = \frac{1}{\sum_{j=0}^{k} (1+2\alpha\mu)^j}$ satisfies
\begin{align*}
    \frac{1}{\beta_{k}} = \sum_{j=0}^{k} (1+2\alpha\mu)^j = 1 + (1+2\alpha\mu) \sum_{j=0}^{k-1} (1+2\alpha\mu)^j = 1 + \frac{1+2\alpha\mu}{\beta_{k-1}} 
\end{align*}
for $k=1,2,\dots$, so $q_k = \frac{1}{(1+2\alpha\mu)^{k-1}} \frac{1}{\beta_{k-1}}$ satisfies the recurrence relation
\begin{align}
    q_{k+1} = \frac{q_k}{\beta_k} \frac{\beta_{k-1}}{(1+2\alpha\mu)} =
    \frac{q_k}{\beta_k} \frac{1}{\frac{1}{\beta_k} - 1} = \frac{q_k}{1-\beta_k}
    \label{eqn:sm-EAG_recurrence_q_k}
\end{align}
for $k=1,2,\dots$.
Additionally, because $\eta_k = \frac{1-\beta_k}{1+2\alpha\mu}$ and $p_k = \frac{\eta_k \alpha q_k}{2\beta_k}$,
\begin{align}
    p_k = \frac{\alpha (1-\beta_k) q_k}{2\beta_k (1+2\alpha\mu)} = \frac{\alpha}{2} \left( \frac{1}{\beta_k} - 1 \right) \frac{q_k}{1 + 2\alpha\mu} = \frac{\alpha}{2} \frac{1}{\beta_{k-1}} q_k = \frac{\alpha}{2} (1+2\alpha\mu)^{k-1} q_k^2 
    \label{eqn:sm-EAG_pkqk}
\end{align}
and
\begin{align}
    p_{k+1} & = \frac{\eta_{k+1}\alpha q_{k+1}}{2\beta_{k+1}} = \frac{\alpha q_{k+1}}{2\beta_{k+1}} \frac{1-\beta_{k+1}}{1+2\alpha\mu} = \frac{\alpha q_k}{2(1-\beta_k)} \frac{1-\beta_{k+1}}{(1+2\alpha\mu)\beta_{k+1}} \nonumber \\
    & = \frac{\alpha q_k}{2(1-\beta_k)} \frac{1}{1+2\alpha\mu} \left( \frac{1}{\beta_{k+1}} - 1 \right) = \frac{\alpha q_k}{2\beta_k (1-\beta_k)} . 
    \label{eqn:sm-EAG_recurrence_p_k}
\end{align}
The identities~\eqref{eqn:sm-EAG_recurrence_q_k} through \eqref{eqn:sm-EAG_recurrence_p_k} will be useful for the subsequent arguments.

Using $\mu$-strong monotonicity of $\opB$ and the identities~\eqref{eqn:sm_EAG+_zkp1_minus_z0}, \eqref{eqn:SM-EAG+_zk_minus_zkp1} and \eqref{eqn:sm-EAG_recurrence_q_k}, we have
\begin{align}
    V_k - V_{k+1} & \ge V_k - V_{k+1} - \frac{q_k}{\beta_k} \langle z_k - z_{k+1}, \opB z_k - \opB z_{k+1} - \mu(z_k - z_{k+1}) \rangle \nonumber \\
    & = p_k \|\opB z_k\|^2 + q_k \langle \opB z_k , z_k-z_0 \rangle - q_k \mu \|z_k - z_0\|^2 \nonumber \\
    & \quad - p_{k+1} \|\opB z_{k+1}\|^2 - q_{k+1} \langle \opB z_{k+1} , (1-\beta_k)(z_k - z_0) - \alpha \opB z_{k+1/2} \rangle \nonumber \\
    & \quad + q_{k+1} \mu \left\| (1-\beta_k) (z_k - z_0) - \alpha \opB z_{k+1/2} \right\|^2 \nonumber \\
    & \quad - \frac{q_k}{\beta_k} \langle \beta_k (z_k - z_0) + \alpha \opB z_{k+1/2} , \opB z_k - \opB z_{k+1} \rangle \nonumber \\
    & \quad + \frac{q_k \mu}{\beta_k} \left\| \beta_k (z_k - z_0) + \alpha \opB z_{k+1/2} \right\|^2 \nonumber \\
    & \begin{aligned}
        & = p_k \|\opB z_k\|^2 - \frac{\alpha q_k}{\beta_k} \langle \opB z_k, \opB z_{k+1/2} \rangle + \mu \alpha^2 \left( \frac{q_k}{\beta_k} + q_{k+1} \right) \|\opB z_{k+1/2}\|^2 \\
        & \quad + \alpha \left( \frac{q_k}{\beta_k} + q_{k+1} \right) \langle \opB z_{k+1/2} , \opB z_{k+1} \rangle - p_{k+1} \|\opB z_{k+1}\|^2 \\
        & \quad + 2\mu\alpha \left( q_k - (1-\beta_k)q_{k+1} \right) \langle \opB z_{k+1/2}, z_k - z_0 \rangle \\
        & \quad + (q_k - (1-\beta_k) q_{k+1}) \langle \opB z_{k+1}, z_k - z_0 \rangle \\
        & \quad + \mu (1-\beta_k) ((1-\beta_k) q_{k+1} - q_k) \|z_k - z_0\|^2 
    \end{aligned} \nonumber \\
    & \begin{aligned}
        & = p_k \|\opB z_k\|^2 - \frac{\alpha q_k}{\beta_k} \langle \opB z_k, \opB z_{k+1/2} \rangle + \frac{\mu \alpha^2 q_k}{\beta_k (1-\beta_k)} \|\opB z_{k+1/2}\|^2 \\
        & \quad + \frac{\alpha q_k}{\beta_k (1-\beta_k)} \langle \opB z_{k+1/2} , \opB z_{k+1} \rangle - p_{k+1} \|\opB z_{k+1}\|^2 .
    \end{aligned} \label{eqn:sm-EAG_Lyapunov_monotone}
\end{align}
Now because $\opB$ is $L$-Lipschitz,
\begin{align}
\label{eqn:sm-EAG_Lyapunov_Lipschitz}
    \|\opB z_{k+1/2} - \opB z_{k+1}\|^2 \le L^2 \|z_{k+1/2} - z_{k+1}\|^2 = L^2 \|\alpha \opB z_{k+1/2} - \eta_k \alpha \opB z_k\|^2 
\end{align}
where we applied the identity~\eqref{eqn:sm_EAG+_zkhalf_minus_zkp1} for the last equality.
Multiplying $p_{k+1}$ to \eqref{eqn:sm-EAG_Lyapunov_Lipschitz} and then subtracting it from \eqref{eqn:sm-EAG_Lyapunov_monotone}, we obtain
\begin{align}
    V_k - V_{k+1} & \ge ( p_k - p_{k+1} \eta_k^2 \alpha^2 L^2) \|\opB z_k\|^2 - \left( \frac{\alpha q_k}{\beta_k} - 2p_{k+1} \eta_k \alpha^2 L^2 \right) \langle \opB z_k, \opB z_{k+1/2} \rangle \nonumber \\
    & \quad + \left( \frac{\mu\alpha^2 q_k}{\beta_k (1-\beta_k)} + p_{k+1}(1-\alpha^2 L^2) \right) \|\opB z_{k+1/2}\|^2 \nonumber \\
    & \quad + \left( \frac{\alpha q_k}{\beta_k (1-\beta_k)} - 2p_{k+1} \right) \langle \opB z_{k+1/2} , \opB z_{k+1} \rangle \nonumber \\
    & = \frac{\alpha (1 + 2\alpha\mu - \alpha^2 L^2) q_k}{2\beta_k (1-\beta_k)} \|\eta_k \opB z_k - \opB z_{k+1/2}\|^2
    \label{eqn:sm-EAG_summable_term}
\end{align}
where the last equality follows from the identities $\eta_k = \frac{1-\beta_k}{1+2\alpha\mu}$, $p_k = \frac{\eta_k \alpha q_k}{2\beta_k}$ and \eqref{eqn:sm-EAG_recurrence_p_k}.
As $1 + 2\alpha\mu - \alpha^2 L^2 \ge 0$ if $0 < \alpha \le \frac{\sqrt{L^2 + \mu^2} + \mu}{L^2}$, for this range of $\alpha$,
\begin{align*}
    V_k - V_{k+1} & \ge \frac{\alpha (1 + 2\alpha\mu - \alpha^2 L^2)}{2} (1+2\alpha\mu)^k \|\eta_k \opB z_k - \opB z_{k+1/2}\|^2 \ge 0,
\end{align*}
where we used $\frac{q_k}{\beta_k (1-\beta_k)} = \frac{q_{k+1}}{\beta_k} > (1+2\alpha\mu)^k$.
Now it only remains to show that $V_N \ge 0$ holds for any $N \in \mathbb{N}$; if that is the case, then
\begin{align*}
    & \sum_{k=0}^N (1+2\alpha\mu)^k \|\eta_k \opB z_k - \opB z_{k+1/2}\|^2 \le \frac{2}{\alpha (1 + 2\alpha\mu - \alpha^2 L^2)} (V_0 - V_{N+1}) \\
    & \quad \le \frac{2}{\alpha (1 + 2\alpha\mu - \alpha^2 L^2)} \left( \frac{1}{2\alpha} + \mu \right) \|z_0 - z_\star\|^2 = \frac{1+2\alpha\mu}{\alpha^2 (1 + 2\alpha\mu - \alpha^2 L^2)} \|z_0 - z_\star\|^2 ,
\end{align*}
and taking $N\to\infty$, we obtain~\eqref{eqn:sm-EAG_summability}.

Consider the following shifted version of $V_k$:
\begin{align*}
    W_k = V_k - \left(\frac{1}{2\alpha} + \mu\right) \|z_0 - z_\star\|^2 = p_k \|\opB z_k\|^2 + q_k \langle \opB z_k - \mu (z_k - z_0), z_k - z_0 \rangle .
\end{align*}
Because $\opB$ is $\mu$-strongly monotone, if $\lambda \ge 0$ and $\lambda \ne q_k$,
\begin{align}
    W_k & \ge p_k \|\opB z_k\|^2 + q_k \langle \opB z_k - \mu (z_k - z_0), z_k - z_0 \rangle - \lambda \langle \opB z_k - \mu (z_k - z_\star), z_k - z_\star \rangle \nonumber \\
    &
    \begin{aligned}
        & = p_k \|\opB z_k\|^2 + \langle \opB z_k , (q_k - \lambda) z_k - q_k z_0 + \lambda z_\star \rangle \\
        & \quad + \mu(\lambda-q_k) \|z_k\|^2 + 2 \mu \langle z_k, q_k z_0 - \lambda z_\star \rangle - q_k \mu \|z_0\|^2 + \lambda\mu \|z_\star\|^2  
    \end{aligned} \nonumber \\
    &
    \begin{aligned}
        & = p_k \|\opB z_k\|^2 + \langle \opB z_k , (q_k - \lambda) z_k - q_k z_0 + \lambda z_\star \rangle \\
        & \quad + \frac{\mu}{\lambda - q_k} \left\| (\lambda - q_k) z_k + q_k z_0 - \lambda z_\star \right\|^2 - \frac{q_k \lambda \mu}{\lambda - q_k} \|z_0 - z_\star\|^2 .
    \end{aligned} \nonumber \\ 
    &
    \begin{aligned}
        & = \left( p_k - \frac{\lambda - q_k}{4\mu} \right) \|\opB z_k\|^2 + \frac{\mu}{\lambda - q_k} \left\| \frac{\lambda - q_k}{2\mu} \opB z_k - (\lambda - q_k) z_k - q_k z_0 + \lambda z_\star \right\|^2 \\
        & \quad - \frac{q_k \lambda \mu}{\lambda - q_k} \|z_0 - z_\star\|^2 .
    \end{aligned} \nonumber
\end{align}
Adding $\frac{q_k \lambda \mu}{\lambda - q_k} \|z_0 - z_\star\|^2$ to the both sides and taking $\lambda = q_k + 4\mu p_k$, we obtain 
\begin{align*}
    0 & \le W_k + \left( \frac{q_k^2}{4p_k} + \mu q_k \right) \|z_0 - z_\star\|^2 = V_k ,
\end{align*}
where the last equality follows by applying \eqref{eqn:sm-EAG_pkqk} and then plugging in the expression $q_k = \frac{1}{(1+2\alpha\mu)^{k-1}} \frac{1}{\beta_{k-1}} = \frac{1}{(1+2\alpha\mu)^{k-1}} \frac{(1+2\alpha\mu)^k - 1}{2\alpha\mu}$.

\subsection{Proof of \texorpdfstring{\cref{lem:oc-halpern-dist-to-sol-rate}}{Lemma 4.5}}

Note that \eqref{eqn:oc_halpern_dist_to_sol_wkp1} is an immediate consequence of \cref{eqn:oc_halpern_dist_to_sol_wkhalf} because
\begin{align*}
    \|w_{k+1} - w_\star\|^2 = \|\opT(w_{k+1/2}) - w_\star\|^2 \le \frac{1}{\gamma^2} \|w_{k+1/2} - w_\star\|^2 .
\end{align*}
When $k=0$, we have $w_{1/2} = w_0$, so both \eqref{eqn:oc_halpern_dist_to_sol_wkhalf} and \eqref{eqn:oc_halpern_dist_to_sol_wkp1} hold trivially.
Now assume that \eqref{eqn:oc_halpern_dist_to_sol_wkhalf} and \eqref{eqn:oc_halpern_dist_to_sol_wkp1} holds for some $k \ge 0$.
To simplify the notation, we write $\left( \frac{\sum_{j=0}^k \gamma^j}{\sum_{j=0}^k \gamma^{2j}} \right)^2 = \rho_k$.
Then 
\begin{align}
    \|w_{k+3/2} - w_\star\|^2 & = \left\| \beta_{k+1} w_0 + (1-\beta_{k+1}) w_{k+1} - w_\star \right\|^2 \nonumber \\
    & = \left\| \beta_{k+1} (w_0 - w_\star) + (1-\beta_{k+1}) (w_{k+1} - w_\star) \right\|^2 \nonumber \\
    & \le \beta_{k+1}^2 \left( 1 + \frac{1}{\varphi_k} \right) \left\| w_0 - w_\star \right\|^2 + (1-\beta_{k+1})^2 (1 + \varphi_k) \left\| w_{k+1} - w_\star \right\|^2 \nonumber \\
    & \le \left[ \beta_{k+1}^2 \left( 1 + \frac{1}{\varphi_k} \right) + \frac{1}{\gamma^2} (1-\beta_{k+1})^2 (1 + \varphi_k) \rho_k \right] \left\| w_0 - w_\star \right\|^2 \label{eqn:oc-halpern-dist-to-sol-recursion}
\end{align}
where the third line uses Young's inequality and $\varphi_k > 0$ can be chosen arbitrarily, and the last line uses the induction hypothesis.
Now we minimize the coefficient of $\|w_0 - w_\star\|^2$ from \cref{eqn:oc-halpern-dist-to-sol-recursion} by choosing $\varphi_k = \frac{\beta_{k+1}\gamma}{(1-\beta_{k+1})\rho_k}$, which gives
\begin{align}
\label{eqn:oc-halpern-dist-to-sol-rate-final}
    \|w_{k+3/2} - w_\star\|^2 \le \frac{\left( (1-\beta_{k+1}) \sqrt{\rho_k} + \beta_{k+1} \gamma \right)^2}{\gamma^2} \|w_0 - w_\star\|^2 .
\end{align}
Finally, we substitute the following into \cref{eqn:oc-halpern-dist-to-sol-rate-final}:
\begin{align*}
    (1-\beta_{k+1}) \sqrt{\rho_k} + \beta_{k+1} \gamma & = \frac{\gamma^2 \sum_{j=0}^k \gamma^{2j}}{\sum_{j=0}^{k+1} \gamma^{2j}} \frac{\sum_{j=0}^k \gamma^j}{\sum_{j=0}^k \gamma^{2j}} + \frac{\gamma}{\sum_{j=0}^{k+1} \gamma^{2j}} \\
    & = \frac{\gamma^2 \sum_{j=0}^k \gamma^j}{\sum_{j=0}^{k+1} \gamma^{2j}} + \frac{\gamma}{\sum_{j=0}^{k+1} \gamma^{2j}} \\
    & = \gamma \frac{\sum_{j=0}^{k+1} \gamma^j}{\sum_{j=0}^{k+1} \gamma^{2j}}
\end{align*}
to complete the induction and conclude the proof.

\section{Omitted proofs for \texorpdfstring{\Cref{sec:composite_minimax}}{Section 5}}
\label{sec:composite_minimax_proofs}

\subsection{Proof of \texorpdfstring{\cref{thm:APG_MP}}{Theorem 5.1}}

We first provide some preliminary results.

\begin{lemma}[{\cite[Lemma~3]{Tran-DinhLuo2021_halperntype}}]
\label{lem:fb_residual_key_inequality}
Suppose that $\opA, \opB$ are maximal monotone operators on $\reals^d$, and $\opB$ is single-valued.
For $\alpha > 0$, let $\opG_\alpha (v) = \frac{1}{\alpha} \left( v - \JA (v-\alpha\opB v) \right)$.
Then for any $v,w \in \reals^d$,
\begin{align*}
    \left\langle \opG_\alpha (v) - \opG_\alpha (w), (\opI + \alpha\opB)(v) - (\opI + \alpha\opB)(w) \right\rangle \ge \alpha \|\opG_\alpha (v) - \opG_\alpha (w)\|^2 .
\end{align*}
\end{lemma}

\begin{lemma}
\label{lem:fb_error_bound}
Let $\{(u_k, w_k)\}_{k=0,1,\dots,}$ be the iterates of \ref{eqn:DRS-Halpern_alternate_form} and $\{(\xi_k, z_k)\}_{k=0,1,\dots,}$ be the iterates of \ref{eqn:APG*}. 
If $u_0 = \xi_0$, then for $k=1,2,\dots$, we have
\begin{align}
\label{eqn:APG_xi_u_difference_bound}
    (k+1) \|\xi_{k} - u_{k}\| \le k \|z_{k-1} + \alpha \opB z_{k-1} - u_{k-1}\|
\end{align}
and
\begin{align}
\label{eqn:APG_z_u_difference_bound}
    (k+1)\|z_{k} + \alpha\opB z_{k} - u_{k}\| & \le k \|z_{k-1} + \alpha\opB z_{k-1} - u_{k-1}\| + (k+1)\epsilon_{k} . 
\end{align}
where each $\epsilon_k$ is an upper bound on $\|z_{k} + \alpha \opB z_{k} - \xi_{k}\|$.
\end{lemma}

\begin{proof}
Observe that $\JA(z - \alpha \opB z) + \alpha\opB z = z + \alpha \opB z - \alpha \opG_\alpha (z)$ for any $z\in \reals^d$, so
\begin{align*}
    \xi_{k} = \frac{1}{k+1} \xi_0 + \frac{k}{k+1} (z_{k-1} + \alpha \opB z_{k-1} - \alpha \opG_\alpha (z_{k-1})).
\end{align*}
Similarly, we have
\begin{align*}
    u_{k} 
    & = \frac{1}{k+1} u_0 + \frac{k}{k+1} (w_{k-1} + \alpha \opB w_{k-1} - \alpha \opG_\alpha (w_{k-1})) .
\end{align*}
Therefore, provided that $u_0 = \xi_0$,
\begin{align*}
    \xi_{k} - u_{k} = \frac{k}{k+1} \left( (\opI + \alpha\opB) (z_{k-1}) - (\opI + \alpha\opB) (w_{k-1}) - \alpha(\opG_\alpha (z_{k-1}) - \opG_\alpha (w_{k-1})) \right) ,
\end{align*}
which gives
\begin{align*}
    \|\xi_{k} - u_{k}\|^2 & = \left(\frac{k}{k+1}\right)^2 \left\| (\opI + \alpha\opB) (z_{k-1}) - (\opI + \alpha\opB) (w_{k-1}) - \alpha(\opG_\alpha (z_{k-1}) - \opG_\alpha (w_{k-1})) \right\|^2 \\
    & = \left(\frac{k}{k+1}\right)^2 \big( \|(\opI + \alpha\opB)(z_{k-1}) - (\opI + \alpha\opB)(w_{k-1})\|^2  + \alpha^2 \|\opG_\alpha(z_{k-1}) - \opG_\alpha(w_{k-1})\|^2 \\
    & \quad \quad \quad \quad \quad \quad \quad - 2\alpha \left\langle (\opI + \alpha\opB)(z_{k-1}) - (\opI + \alpha\opB)(w_{k-1}) , \opG_\alpha (z_{k-1}) - \opG_\alpha (w_{k-1}) \right\rangle \big) \\
    & \le \left(\frac{k}{k+1}\right)^2 \big( \|(\opI + \alpha\opB)(z_{k-1}) - (\opI + \alpha\opB)(w_{k-1})\|^2  - \alpha^2 \|\opG_\alpha(z_{k-1}) - \opG_\alpha(w_{k-1})\|^2 \big) \\
    & \le \left(\frac{k}{k+1}\right)^2 \|(\opI + \alpha\opB)(z_{k-1}) - (\opI + \alpha\opB)(w_{k-1})\|^2 \\
    & = \left(\frac{k}{k+1}\right)^2 \|(\opI + \alpha\opB)(z_{k-1}) - u_{k-1}\|^2
\end{align*}
where the first inequality follows from \cref{lem:fb_residual_key_inequality} and the last equality from $w_{k-1} = \JB(u_{k-1})$.
This proves~\eqref{eqn:APG_xi_u_difference_bound}.
Then \eqref{eqn:APG_z_u_difference_bound} follows from
\begin{align*}
    \|z_{k} + \alpha\opB z_{k} - u_{k}\| & \le \|z_{k} + \alpha \opB z_{k} - \xi_{k}\| + \|\xi_{k} - u_{k}\| \\
    & \le \epsilon_{k} + \frac{k}{k+1} \|(z_{k-1} + \alpha\opB z_{k-1}) - u_{k-1}\| .
\end{align*}
\end{proof}

We now prove \cref{thm:APG_MP}.

\begin{proof}[Proof of \cref{thm:APG_MP}]

Let $M = 1 + \frac{1}{L}\|\opB \xi_0\|$, so that $\epsilon_j = \frac{M}{(j+1)^2 (j+2)}$ and
\begin{align*}
    M = \sum_{j=0}^\infty (j+1) \epsilon_j .
\end{align*}
Then by \cref{lem:fb_error_bound}, for each $k=1,2,\dots$,
\begin{align*}
    (k+1)\|\xi_{k} - u_{k}\| & \le k\|z_{k-1} + \alpha\opB z_{k-1} - u_{k-1}\| \\
    & \le (k-1) \|z_{k-2} + \alpha\opB z_{k-2} - u_{k-2}\| + k \epsilon_k \\
    & \le \cdots \le \|z_0 + \alpha\opB z_0 - u_0\| + \sum_{j=1}^k (j+1) \epsilon_j \\
    & \le \sum_{j=0}^k (j+1) \epsilon_j \le M .
\end{align*}
This shows that $\|\xi_k - u_k\|^2 \le \left(\frac{M}{k+1}\right)^2$ for $k=0,1,\dots$ (note that $\|\xi_0 - u_0\| = 0$).
Additionally, for each $k=0,1,\dots$ we have $(k+1)\|z_k + \alpha\opB z_k - u_k\| \le M$, and because $\JB = (\opI + \alpha\opB)^{-1}$ is nonexpansive, we have
\begin{align*}
    (k+1)\|z_k - w_k\| & = (k+1)\|\JB \circ (\opI + \alpha\opB)(z_k) - \JB(u_k)\| \\
    & \le (k+1)\|z_k + \alpha\opB z_k - u_k\| \le M ,
\end{align*}
i.e., $\|z_k - w_k\| \le \left(\frac{M}{k+1}\right)^2$.
Finally, the proof is complete by using
\begin{align*}
    M \le 1 + \frac{1}{L}(\|\opB \xi_\star\| + \|\opB \xi_0 - \opB \xi_\star\|) \le 1 + \frac{1}{L}\|\opB \xi_\star\| + \|\xi_0 - \xi_\star\| = \frac{C(\xi_0)}{L} .
\end{align*}

\end{proof}

\subsection{Proof of \texorpdfstring{\cref{cor:APG_convergence}}{Corollary 5.2}}

As $\TDRS$ is nonexpansive \cite{LionsMercier1979_splitting} and the iterates of OHM converge to the nearest fixed point to the initial point \cite[Theorem~30.1]{BauschkeCombettes2017_convex}, we have $u_k \to \proj_{\mathrm{Fix}(\opT_{\textrm{DRS}})}(u_0) = \proj_{\mathrm{Fix}(\opT_{\textrm{DRS}})}(\xi_0) = \xi_\star$.
Because $\JB$ is nonexpansive, this implies $w_k = \JB(u_k) \to \JB(\xi_\star) \in \zer(\opA + \opB)$.
Due to the MP property proved in \cref{thm:APG_MP}, we immediately obtain $\xi_k \to \xi_\star$ and $z_k \to \JB(\xi_\star)$.

Next, note that \cref{lem:fb_residual_key_inequality} implies that $\opG_\alpha$ is $\left(\frac{1+\alpha L}{\alpha}\right)$-Lipschitz, because
\begin{align*}
    \alpha \|\opG_\alpha (v) - \opG_\alpha (w)\|^2 & \le \left\langle \opG_\alpha (v) - \opG_\alpha (w), (\opI + \alpha\opB)(v) - (\opI + \alpha\opB)(w) \right\rangle \\
    & \le \left\| \opG_\alpha (v) - \opG_\alpha (w) \right\| \left\| (\opI + \alpha\opB)(v) - (\opI + \alpha\opB)(w) \right\| \\
    & \le \left\| \opG_\alpha (v) - \opG_\alpha (w) \right\| (1+\alpha L) \|v-w\|
\end{align*}
for any $v,w\in\reals^d$.
Therefore, we have
\begin{align}
    \|\opG_\alpha (z_k)\|^2 & = \|\opG_\alpha (w_k) + \left(\opG_\alpha (z_k) - \opG_\alpha (w_k)\right)\|^2 \nonumber \\
    & \le (1+\nu) \|\opG_\alpha (w_k)\|^2 + \left(1+\frac{1}{\nu}\right) \|\opG_\alpha (z_k) - \opG_\alpha (w_k)\|^2 \label{eqn:fb_residual_bound_young} \\
    & \le (1+\nu) \frac{4\|\xi_0 - \xi_\star\|^2}{\alpha^2 (k+1)^2} + \left(1+\frac{1}{\nu}\right) \left(\frac{1+\alpha L}{\alpha}\right)^2 \|z_k - w_k\|^2 \label{eqn:fb_residual_bound_OHM_rate_and_Lipschitzness} \\
    & \le (1+\nu) \frac{4C(\xi_0)^2}{\alpha^2 L^2 (k+1)^2} + \left(1+\frac{1}{\nu}\right) \left(\frac{1+\alpha L}{\alpha}\right)^2 \frac{C(\xi_0)^2}{L^2 (k+1)^2} \label{eqn:fb_residual_bound_MP} \\
    & = \frac{C(\xi_0)^2}{\alpha^2 L^2 (k+1)^2} \left( 4(1+\nu) + \left(1+\frac{1}{\nu}\right) (1+\alpha L)^2 \right) \label{eqn:fb_residual_bound_nu} 
\end{align}
where $\nu>0$ is arbitrary, \eqref{eqn:fb_residual_bound_young} uses Young's inequality, \eqref{eqn:fb_residual_bound_OHM_rate_and_Lipschitzness} uses the convergence rate of OHM and Lipschitzness of $\opG_\alpha$, and \eqref{eqn:fb_residual_bound_MP} uses $L\|\xi_0-\xi_\star\| \le C(\xi_0)$ and \cref{thm:APG_MP}.
Now choosing $\nu=\frac{1+\alpha L}{2}$, which minimizes the quantity \eqref{eqn:fb_residual_bound_nu}, we obtain
\begin{align*}
    \|\opG_\alpha (z_k)\|^2 \le \frac{C(\xi_0)^2}{\alpha^2 L^2 (k+1)^2} \left( 4 + (1+\alpha L)^2 + 4(1+\alpha L) \right) = \frac{(3+\alpha L)^2 C(\xi_0)^2}{\alpha^2 L^2 (k+1)^2} ,
\end{align*}
which completes the proof.

\subsection{Proof of \texorpdfstring{\cref{cor:APG_complexity}}{Corollary 5.3}}

We first state and prove some lemmas, which will be useful for the main proof.

\begin{lemma}
\label{lem:fixed_point_residual_bound}
Let $\opB\colon \reals^d \to \reals^d$ be a monotone, $L$-Lipschitz operator, and let $\alpha \in \left(0,\frac{1}{L}\right)$.
Then for any $u \in \reals^d$,
\begin{align*}
    \|u - \JB(u)\| \le \frac{\alpha}{1-\alpha L}\|\opB u\| .
\end{align*}
\end{lemma}

\begin{proof}
Let $v = \JB(u)$. Then $u = v + \alpha \opB v$, so
\begin{align*}
    \|\opB v\| - \|\opB u\| \le \|\opB v - \opB u\| \le L\|v - u\| = \alpha L\|\opB v\| \implies \|\opB v\| \le \frac{1}{1-\alpha L}\|\opB u\| .
\end{align*}
Therefore,
\begin{align*}
    \|u - \JB(u)\| = \|u - v\| = \alpha \|\opB v\| \le \frac{\alpha}{1-\alpha L} \|\opB u\| .
\end{align*}
\end{proof}

\begin{lemma}
\label{lem:iterate_boundedness}
Let $\xi_\star \in \fix\TDRS$ and $z_\star = \JB(\xi_\star)$.
Let $\epsilon_k > 0$ ($k=0,1,\dots$) satisfy
\begin{align*}
    \sum_{k=0}^\infty (k+1)\epsilon_k = M < \infty .
\end{align*}
Then \ref{eqn:APG*} iterates $\xi_k, z_k$ satisfy $\max\left\{ \|\xi_k - \xi_\star\|, \|z_k - z_\star\| \right\} \le \|\xi_0 - \xi_\star\| + M$ for $k=0,1,\dots$. 
\end{lemma}

\begin{proof}
Let $u_k, w_k$ be the iterates of~\ref{eqn:DRS-Halpern_alternate_form} with $u_0 = \xi_0$.
Because $\xi_\star \in \mathrm{Fix}(\opT_{\textrm{DRS}})$ and $\opT_{\textrm{DRS}}$ is nonexpansive, we have 
\begin{align*}
    \|u_{k+1} - \xi_\star\| & = \left\| \beta_k u_0 + (1-\beta_k) \TDRS(u_k) - \xi_\star \right\| \\
    & \le \beta_k \|u_0 - \xi_\star\| + (1-\beta_k) \|\opT_{\textrm{DRS}}(u_k) - \xi_\star\| \\
    & \le \beta_k \|\xi_0 - \xi_\star\| + (1-\beta_k) \|u_k - \xi_\star\|
\end{align*}
for $k=0,1,\dots$. 
Therefore, by induction on $k$, $\|u_k-\xi_\star\| \le \|\xi_0-\xi_\star\|$ for all $k=0,1,\dots$.
Then by \cref{thm:APG_MP},
\begin{align*}
    \|\xi_k - \xi_\star\| \le \|\xi_k - u_k\| + \|u_k - \xi_\star\| \le \frac{M}{k+1} + \|\xi_0 - \xi_\star\| \le \|\xi_0 - \xi_\star\| + M 
\end{align*}
and
\begin{align*}
    \|z_k - z_\star\| & \le \|z_k - w_k\| + \|w_k - z_\star\| \le \frac{M}{k+1} + \|\JB(u_k) - \JB(\xi_\star)\| \\
    & \le \frac{M}{k+1} + \|u_k - \xi_\star\| \le \|\xi_0 - \xi_\star\| + M 
\end{align*}
for $k=0,1,\dots$.
\end{proof}

We now prove \cref{cor:APG_complexity}.

\begin{proof}[Proof of \cref{cor:APG_complexity}]

Recall that each outer loop of \ref{eqn:APG*} performs a single evaluation of $\JA$, a forward evaluation $\opB z_k$, and an inner loop, which starts with $z_k^{(0)} = \xi_{k}$ and solves 
\begin{align}
    \begin{array}{ll}
        \underset{z \in \reals^{d}}{\textrm{find}} & 0 = z + \alpha \opB z - \xi_k
    \end{array}
    \label{eqn:SASH-DRS_inner_loop_appendix}
\end{align}
using \ref{eqn:SM-EAG+}.
The distance from the initial point $z_k^{(0)}$ to the true solution $\JB(\xi_{k})$ of~\eqref{eqn:SASH-DRS_inner_loop_appendix} is 
\begin{align*}
    d_{k}^{(0)} & := \|z_k^{(0)} - \JB(\xi_k)\| = \|\xi_k - \JB(\xi_k)\| \le \frac{\alpha}{1-\alpha L} \|\opB \xi_k\| 
\end{align*}
where the inequality holds due to \cref{lem:fixed_point_residual_bound}.
Observe that by $L$-Lipschitzness of $\opB$ and \cref{lem:iterate_boundedness}, we have
\begin{align*}
    \|\opB \xi_k\| \le \|\opB \xi_\star\| + \|\opB \xi_k - \opB \xi_\star\| \le \|\opB \xi_\star\| + L \|\xi_k - \xi_\star\| \le \|\opB \xi_\star\| + L \|\xi_0 - \xi_\star\| + LM ,
\end{align*}
so $d_k^{(0)} \le \frac{\alpha}{1-\alpha L} \left(\|\opB \xi_\star\| + L \|\xi_0 - \xi_\star\| + LM\right)$.
Because the inner loop objective~\eqref{eqn:SASH-DRS_inner_loop_appendix} is $1$-strongly monotone and $(\alpha L+1)$-Lipschitz, by \cref{thm:sm-EAG}, the total number of inner loop iterations needed to achieve the error $\|z+\alpha\opB z - \xi_k\| \le \epsilon_{k}$ is
\begin{align*}
    \mathcal{O}\left( (\alpha L+1) \log \frac{d_{k}^{(0)}}{\epsilon_{k}} \right) &= \mathcal{O} \left( (\alpha L+1) \log \frac{\frac{\alpha}{1-\alpha L} \left(\|\opB \xi_\star\| + L \|\xi_0 - \xi_\star\| + LM\right)}{M (k+1)^{-2}(k+2)^{-1}} \right) \\
    &= \mathcal{O}\left( \log \frac{k (\|\opB \xi_0\| + 2L\|\xi_0 - \xi_\star\| + LM)}{LM} \right) 
\end{align*}
where the last $\cO$ expression holds because $\alpha L \in (0,1)$ is a constant.
Now recalling that $M = 1 + \frac{1}{L}\|\opB \xi_0\|$, we obtain
\begin{align*}
    \log \frac{k (\|\opB \xi_0\| + 2L\|\xi_0 - \xi_\star\| + LM)}{LM}
    & \le \log \left( k \left( \frac{\|\opB \xi_0\|}{LM} + \frac{2\|\xi_0-\xi_\star\|}{M} + 1 \right) \right) \\
    & \le \log \left( k \left( 1 + 2\|\xi_0-\xi_\star\| + 1 \right) \right) \\
    & = \mathcal{O} ( \log (k(\|\xi_0-\xi_\star\|+1)) ) .
\end{align*}
By \cref{cor:APG_convergence}, it requires at most $\mathcal{O}\left( \frac{C(\xi_0)}{\epsilon} \right) = \mathcal{O}\left(\frac{L(\|\xi_0-\xi_\star\|+1)+\|\opB \xi_\star\|}{\epsilon}\right)$ to achieve $\|\opG_\alpha (z_k)\| \le \epsilon$.
Up to this range of $k$, each inner loop performs
\begin{align*}
    \mathcal{O} ( \log (k(\|\xi_0-\xi_\star\|+1)) ) = \mathcal{O}\left( \log \frac{L(\|\xi_0-\xi_\star\|+1)+\|\opB \xi_\star\|}{\epsilon} \right) = \mathcal{O}\left( \log \frac{C(\xi_0)}{\epsilon} \right)
\end{align*}
evaluations of $\opB$.
Therefore, the total computational cost of \ref{eqn:APG*} is
\begin{align*}
    \mathcal{O}\left( \frac{C(\xi_0)}{\epsilon} \left( C_\opA + C_\opB \log \frac{C(\xi_0)}{\epsilon} \right) \right) .
\end{align*}

\end{proof}

\section{Proof of the remaining summability lemmas}
\label{sec:summability_lemmas}

We first provide a handy lemma, which will be used in the proofs of Lemmas~\ref{lem:EAG_summability} and \ref{lem:APS_summability} for arguing the positivity of certain expressions.

\begin{lemma}
\label{lem:positive_for_small_stepsizes}
Let $N$ be a fixed positive integer.
Suppose $H^{(\ell)}(r)$, $\ell=0,\dots,N$, are continuous functions in $r \in \reals$, and consider the expression
\begin{align*}
    H(r, k) = H^{(N)}(r) k^N + \cdots + H^{(1)}(r) k + H^{(0)}.
\end{align*}
If $H^{(N)}(0) > 0$ and $H^{(\ell)}(0) \ge 0$ for all $\ell=0,\dots,N-1$, then there exists $\eta > 0$ such that $H(r,k) > 0$ for all $r \in (0,\eta)$ and $k=1,2,\dots$.
\end{lemma}

\begin{proof}
Let $J = \left\{\ell \in \{0,\dots,N-1\} \,|\, H^{(\ell)}(0) = 0\right\}$ (so $H^{(\ell)}(r) > 0$ for $\ell \in \{0,\dots,N-1\}\setminus J$).
As $H^{(\ell)}(r)$ are continuous, one can choose $\eta>0$ small enough so that $|r|<\eta$ implies
\begin{align*}
    \sum_{\ell\in J} \left| H^{(\ell)}(r) \right| < \frac{H^{(N)}(0)}{2} 
\end{align*}
and $H^{(\ell)}(r) > 0$ for all $\ell \in \{0,\dots,N-1\}\setminus J$.
Then for $k=1,2,\dots$, we have
\begin{align*}
    H(r,k) & = k^N \left( H^{(N)}(0) + \sum_{\ell=0}^{N-1} H^{(\ell)}(r) k^{\ell-N} \right) \\
    & > k^N \left( H^{(N)}(0) + \sum_{\ell \in J} H^{(\ell)}(r) k^{\ell-N} \right) \\
    & \ge k^N \left( H^{(N)}(0) - \sum_{\ell \in J} \left|H^{(\ell)}(r)\right| k^{\ell-N} \right) \\
    & \ge k^N \left( H^{(N)}(0) - \sum_{\ell \in J} \left|H^{(\ell)}(r)\right| \right) \\
    & > \frac{H^{(N)}(0) k^N}{2} > 0
\end{align*}
provided that $r\in (0,\eta)$.
\end{proof}

\begin{proof}[Proof of Lemma~\ref{lem:EAG_summability}]

We show the given statement by establishing the following result: by defining the sequences $p_k, q_k$ by $q_k = k$ ($k=0,1,\dots$) and
\begin{align*}
    p_0 & = \frac{\alpha (1 - 3\alpha L + 6\alpha^2 L^2 - 2\alpha^4 L^4)}{2 (1-\alpha L)^2 (1+\alpha L) (2+\alpha L)} \\
    p_k & = \frac{\alpha (k+1)(k+\alpha L(k-1))}{2(1+\alpha L)}, \quad k=1,2,\dots ,
\end{align*}
we have $p_0 > 0$,
\begin{align*}
    V_k = p_k \|\opB z_k\|^2 + q_k \langle \opB z_k, z_k - z_0 \rangle + \frac{1}{2\alpha} \|z_0 - z_\star\|^2 \ge 0 ,
\end{align*}
and
\begin{align*}
    V_k - V_{k+1} \ge \epsilon (k+1)^2 \|\opB z_k - \opB z_{k+1/2}\|^2 
\end{align*}
for $\alpha \in \left(0, \frac{\eta}{L}\right)$, $\epsilon = \frac{\alpha^2 L (1-\alpha^2 L^2)}{2}$ with $\eta \in (0,1)$ sufficiently small.
Once this is shown, we can conclude that for any positive integer $N$,
\begin{align*}
    \frac{\alpha^2 L (1 - \alpha^2 L^2)}{2} \sum_{k=0}^N (k+1)^2 \|\opB z_k - \opB z_{k+1/2}\|^2 & \le V_0 - V_{N+1} \\
    & \le V_0 = p_0 \|\opB z_0\|^2 + \frac{1}{2\alpha} \|z_0 - z_\star\|^2 \\
    & \le \left( p_0 L^2 + \frac{1}{2\alpha} \right) \|z_0 - z_\star\|^2 
\end{align*}
holds, which gives
\begin{align*}
    \sum_{k=0}^\infty (k+1)^2 \|\opB z_k - \opB z_{k+1/2}\|^2 \le \frac{2}{\alpha^2 L (1 - \alpha^2 L^2)} \left( p_0 L^2 + \frac{1}{2\alpha} \right) \|z_0 - z_\star\|^2 = \frac{C}{\alpha^2} \|z_0 - z_\star\|^2
\end{align*}
where
\begin{align*}
    C = \frac{2 - r - 2r^2 - 2r^3 + 7r^4 - 2r^6}{r(1-r)^3 (1+r)^2 (2+r)} \quad (r = \alpha L)
\end{align*}
is a positive constant provided that $r < \eta$ for $\eta$ small enough.

\paragraph{Case $k=0$.}
We have $z_{1/2} = z_0 - \alpha \opB z_0$ and $z_1 = z_0 - \alpha \opB z_{1/2}$.
By monotonicity and Lipschitzness of $\opB$, we have
\begin{align}
    & V_1 + \frac{\alpha^2 L (1-\alpha^2 L^2)}{2} \|\opB z_0 - \opB z_{1/2}\|^2 \nonumber \\
    \label{eqn:EAG_C_Lyapunov_initial}
    & = \frac{\alpha}{1+\alpha L}\|\opB z_1\|^2 + \langle \opB z_1, z_1 - z_0 \rangle + \frac{1}{2\alpha} \|z_0 - z_\star\|^2 + \frac{\alpha^2 L (1-\alpha^2 L^2)}{2} \|\opB z_0 - \opB z_{1/2}\|^2 \\
    & \le \frac{\alpha}{1+\alpha L}\|\opB z_1\|^2 + \langle \opB z_1, z_1 - z_0 \rangle + \frac{1}{2\alpha} \|z_0 - z_\star\|^2 + \frac{\alpha^2 L (1-\alpha^2 L^2)}{2} \|\opB z_0 - \opB z_{1/2}\|^2 \nonumber \\
    & \quad + \frac{\alpha}{1+\alpha L} \left( L^2 \|z_{1/2} - z_1\|^2 - \|\opB z_{1/2} - \opB z_1\|^2 \right) + \frac{1-\alpha L}{1+\alpha L} \langle z_0 - z_1, \opB z_0 - \opB z_1 \rangle \nonumber \\
    & = \frac{\alpha}{1+\alpha L}\|\opB z_1\|^2 + \langle \opB z_1, -\alpha\opB z_{1/2} \rangle + \frac{1}{2\alpha} \|z_0 - z_\star\|^2 + \frac{\alpha^2 L (1-\alpha^2 L^2)}{2} \|\opB z_0 - \opB z_{1/2}\|^2 \nonumber \\
    & \quad + \frac{\alpha}{1+\alpha L} \left( \alpha^2 L^2 \|\opB z_0 - \opB z_{1/2}\|^2 - \|\opB z_{1/2} - \opB z_1\|^2 \right) + \frac{1-\alpha L}{1+\alpha L} \langle \alpha \opB z_{1/2}, \opB z_0 - \opB z_1 \rangle \nonumber \\
    &
    \begin{aligned}
    & = \frac{\alpha}{2(1+r)} \big( r (1 + 3r - r^2 - r^3) \|\opB z_0\|^2 - (1-r)^2 (1+r) (2+r) \|\opB z_{1/2}\|^2 \\
    & \quad \quad \quad \quad \quad + 2(1+r)(1 - 3r + r^3) \langle \opB z_0, \opB z_{1/2} \rangle \big) + \frac{1}{2\alpha} \|z_0 - z_\star\|^2
    \end{aligned}
    \label{eqn:EAG_C_summability_initial}
\end{align}
where $r = \alpha L$.
Next, using Young's inequality, we can bound the inner product term in~\eqref{eqn:EAG_C_summability_initial} as
\begin{align*}
    & 2(1+r)(1 - 3r + r^3) \langle \opB z_0, \opB z_{1/2} \rangle \\
    & \le (1-r)^2 (1+r) (2+r) \|\opB z_{1/2}\|^2 + \frac{(1+r)(1-3r+r^3)^2}{(1-r)^2 (2+r)} \|\opB z_0\|^2 .
\end{align*}
Plugging this back into~\eqref{eqn:EAG_C_summability_initial} and rearranging, we obtain
\begin{align*}
    V_1 + \frac{\alpha^2 L (1-\alpha^2 L^2)}{2} \|\opB z_0 - \opB z_{1/2}\|^2 \le \frac{\alpha (1 - 3r + 6r^2 - 2r^4)}{2 (1-r)^2 (1+r) (2+r)} \|\opB z_0\|^2 + \frac{1}{2\alpha} \|z_0 - z_\star\|^2 = V_0 .
\end{align*}
Clearly, if $r$ is small enough, $p_0 = \frac{\alpha (1 - 3r + 6r^2 - 2r^4)}{2 (1-r)^2 (1+r) (2+r)} > 0$.

\paragraph{Case $k\ge 1$.}
From the definition of the EAG iterates, we have
\begin{align*}
    z_{k+1} - z_0 & = (1-\beta_k) (z_k - z_0) - \alpha \opB z_{k+1/2} \\
    z_k - z_{k+1} & = \beta_k (z_k - z_0) + \alpha \opB z_{k+1/2} \\
    z_{k+1/2} - z_{k+1} & = \alpha \opB z_{k+1/2} - \alpha \opB z_k .
\end{align*}
Therefore, as in previous sections, we can lower-bound $V_k - V_{k+1}$ as
\begin{align}
\label{eqn:MP_EAG_Vk-Vkp1_lowerbound}
    & V_k - V_{k+1} \ge V_k - V_{k+1} - \frac{q_k}{\beta_k} \langle z_k - z_{k+1} , \opB z_k - \opB z_{k+1} \rangle \nonumber \\ 
    & \quad = p_k \|\opB z_k\|^2 - \frac{\alpha q_k}{\beta_k} \langle \opB z_k, \opB z_{k+1/2} \rangle + \frac{\alpha q_k}{\beta_k (1-\beta_k)} \langle \opB z_{k+1/2} , \opB z_{k+1} \rangle - p_{k+1} \|\opB z_{k+1}\|^2 \nonumber \\
    & \quad = p_k \|\opB z_k\|^2 - \alpha k(k+1) \langle \opB z_k, \opB z_{k+1/2} \rangle + \alpha (k+1)^2 \langle \opB z_{k+1/2} , \opB z_{k+1} \rangle - p_{k+1} \|\opB z_{k+1}\|^2 .
\end{align}
As $\opB$ is $L$-Lipschitz,
\begin{align*}
    0 & \le L^2 \|z_{k+1/2} - z_{k+1}\|^2 - \|\opB z_{k+1/2} - \opB z_{k+1}\|^2 \\
    & = \alpha^2 L^2 \|\opB z_{k+1/2} - \opB z_k\|^2 - \|\opB z_{k+1/2} - \opB z_{k+1}\|^2 .
\end{align*}
Combining this with \eqref{eqn:MP_EAG_Vk-Vkp1_lowerbound}, we obtain
\begin{align*}
    & V_k - V_{k+1} - \epsilon (k+1)^2 \|\opB z_k - \opB z_{k+1/2}\|^2 \\
    & \ge p_k \|\opB z_k\|^2 - \alpha k(k+1) \langle \opB z_k, \opB z_{k+1/2} \rangle + \alpha (k+1)^2 \langle \opB z_{k+1/2} , \opB z_{k+1} \rangle - p_{k+1} \|\opB z_{k+1}\|^2 \\
    & \quad - \epsilon (k+1)^2 \|\opB z_k - \opB z_{k+1/2}\|^2 - \tau \left( \alpha^2 L^2 \|\opB z_{k+1/2} - \opB z_k\|^2 - \|\opB z_{k+1/2} - \opB z_{k+1}\|^2 \right) \\
    &= (p_k - \alpha^2 L^2 \tau - \epsilon (k+1)^2) \left\|\opB z_k \right\|^2 + (\tau (1-\alpha^2 L^2) - \epsilon (k+1)^2) \left\| \opB z_{k+1/2} \right\|^2 + (\tau - p_{k+1}) \left\|\opB z_{k+1} \right\|^2\\
    & \quad + \left( 2\alpha^2 L^2 \tau + 2\epsilon (k+1)^2 - \alpha k(k+1) \right)\, \left\langle \opB z_k, \opB z_{k+1/2} \right\rangle + \left(\alpha(k+1)^2 - 2 \tau \right) \, \left\langle \opB z_{k+1/2}, \opB z_{k+1} \right\rangle\\
    &= \mathrm{Tr} \left( \mathbf{M}_k \mathbf{S}_k \mathbf{M}_k^\intercal \right),
\end{align*}
where $\tau > 0$, $\mathbf{M}_k := \begin{bmatrix}
\opB z_k & \opB z_{k+1/2} & \opB z_{k+1}
\end{bmatrix}$
and
\begin{gather*}
    \mathbf{S}_k = \begin{bmatrix}
    p_k - \alpha^2 L^2 \tau - \epsilon (k+1)^2                     & \alpha^2 L^2 \tau + \epsilon (k+1)^2 - \frac{\alpha}{2} k(k+1)          & 0 \\
    \alpha^2 L^2 \tau + \epsilon (k+1)^2 - \frac{\alpha}{2} k(k+1) & \tau (1-\alpha^2 L^2) - \epsilon (k+1)^2                                & \frac{1}{2}\alpha(k+1)^2 - \tau \\
    0                                                              & \frac{1}{2}\alpha(k+1)^2 - \tau                                         & \tau - p_{k+1}
    \end{bmatrix} \\
    = \begin{bmatrix}
    s_{11} & s_{12} & 0      \\
    s_{12} & s_{22} & s_{23} \\
    0      & s_{23} & s_{33}
    \end{bmatrix} .
\end{gather*}
Now, we choose $\tau$ as
\begin{align*}
    \tau & = \frac{(k+1)\left(k(1-\alpha L + \alpha^2 L^2 + \alpha^3 L^3) - \alpha L(2 - \alpha L - \alpha^2 L^2) \right)}{4\alpha L^2} .
\end{align*}
For notational simplicity, let $\alpha L = r$.
Then by direct computation we obtain
\begin{gather*}
    s_{11} = \frac{\alpha}{4(1+r)} \left( (1 - r^2)^2 k^2 + (1 - 2r - 3r^2 + 2r^4) k - r (2+r-r^3) \right) \\
    s_{12} = - \frac{\alpha (1-r) (k+1) \left( (1-r^2)k - r^2 \right) }{4} .
\end{gather*}
By Lemma~\ref{lem:positive_for_small_stepsizes}, $s_{11} > 0$ for sufficiently small $r$, and in that case, Young's inequality gives
\begin{align*}
    s_{11} \|\opB z_k\|^2 + 2s_{12} \langle \opB z_k , \opB z_{k+1/2} \rangle + \frac{s_{12}^2}{s_{11}} \|\opB z_{k+1/2}\|^2 \ge 0.
\end{align*}
Continuing on, we compute
\begin{align*}
    t_{22} & = s_{22} - \frac{s_{12}^2}{s_{11}} \\
    & = \frac{(1-r^2) \left( T_{22}^{(3)}(r) k^3 + T_{22}^{(2)}(r) k^2 + T_{22}^{(1)}(r) k + T_{22}^{(0)}(r) \right) }
   {4Lr \left( \left(1-r^2\right)^2 k - r \left( 2 + r - r^3 \right)\right)}
\end{align*}
where $T_{22}^{(\ell)}(r)$, $\ell=0,1,2,3$, are polynomials in $r$ defined by
\begin{align*}
	T_{22}^{(0)}(r) & = r^2 (4 + r^2 - r^3) \\
	T_{22}^{(1)}(r) & = - r \left( 4 - 6r - 2r^2 - 3r^3 + 3r^4 \right) \\
	T_{22}^{(2)}(r) & = 1 -5r + 4r^3 + 3r^4 - 3r^5 \\
	T_{22}^{(3)}(r) & = (1+r)^2 (1-r)^3 ,
\end{align*}
and
\begin{align*}
    s_{23} = - \frac{(1+r)(k+1) \left( (1-r)^2 k - r(2-r) \right)}{4Lr}
\end{align*}
By Lemma~\ref{lem:positive_for_small_stepsizes} again, $t_{22} > 0$ for sufficiently small $r$ and then Young's inequality implies that
\begin{align*}
    t_{22} \|\opB z_{k+1/2}\|^2 + 2s_{23} \langle \opB z_{k+1/2} , \opB z_{k+1} \rangle + \frac{s_{23}^2}{t_{22}} \|\opB z_{k+1}\|^2 \ge 0 . 
\end{align*}
Finally, we have
\begin{align*}
    & \mathrm{Tr} \left( \mathbf{M}_k \mathbf{S}_k \mathbf{M}_k^\intercal \right) \\
    & = \left( s_{11} \|\opB z_k\|^2 + 2s_{12} \langle \opB z_k , \opB z_{k+1/2} \rangle + \frac{s_{12}^2}{s_{11}} \|\opB z_{k+1/2}\|^2 \right) \\
    & \quad + \left( t_{22} \|\opB z_{k+1/2}\|^2 + 2s_{23} \langle \opB z_{k+1/2} , \opB z_{k+1} \rangle + \frac{s_{23}^2}{t_{22}} \|\opB z_{k+1}\|^2 \right) + \left( s_{33} - \frac{s_{23}^2}{t_{22}} \right) \|\opB z_{k+1}\|^2 \\
    & \ge 0 ,
\end{align*}
as the following computation and Lemma~\ref{lem:positive_for_small_stepsizes} shows that $s_{33} - \frac{s_{23}^2}{t_{22}} > 0$ for sufficiently small $r$:
\begin{align*}
    s_{33} - \frac{s_{23}^2}{t_{22}} = \frac{\alpha r\left( (1 - 4r - 2r^2 + r^4) k - r(6 - r - r^3) \right)}{2(1-r^2) \left( (1+r)^2 (1-r) k - r(2 + r + r^2) \right)} .
\end{align*}
Therefore, we have established $V_k - V_{k+1} \ge 0$ for small $r$.

\paragraph{Nonnegativity.}
Observe that
\begin{align*}
    p_k - \frac{\alpha k^2}{2} = \frac{\alpha (k+1)(k+\alpha L(k-1)) - \alpha(1+\alpha L)k^2}{2(1+\alpha L)} = \frac{\alpha (k-\alpha L)}{2(1+\alpha L)} > 0
\end{align*}
provided that $k\ge 1$ and $\alpha < \frac{1}{L}$. Thus,
\begin{align*}
    V_k & \ge \frac{\alpha k^2}{2} \|\opB z_k\|^2 + k \langle \opB z_k, z_k - z_0 \rangle + \frac{1}{2\alpha}\|z_0 - z_\star\|^2 \\
    & \ge \frac{\alpha k^2}{2} \|\opB z_k\|^2 + k \langle \opB z_k, z_\star - z_0 \rangle + \frac{1}{2\alpha}\|z_0 - z_\star\|^2 \ge 0
\end{align*}
by monotonicity and Young's inequality.

\end{proof}

\begin{proof}[Proof of Lemma~\ref{lem:APS_summability}]
We establish the following Lyapunov analysis: with same $p_k, q_k$ as in Lemma~\ref{lem:EAG_summability},
\begin{align*}
    V_k = p_k \|\opB z_k\|^2 + q_k \langle \opB z_k, z_k - z_0 \rangle + p_k L^2 \|z_k - v_k\|^2 + \frac{1}{2\alpha} \|z_0 - z_\star\|^2
\end{align*}
satisfies $V_k \ge 0$ and $V_k - V_{k+1} \ge \epsilon(k+1)^2 \|\opB v_k - \opB v_{k+1}\|^2$ for $k=0,1,\dots$, $\alpha \in \left(0, \frac{\eta}{L}\right)$, and $\epsilon = \frac{\alpha^2 L \left( 1 - \alpha L - \alpha^2 L^2 - \alpha^3 L^3 \right)}{2(1 + \alpha L)}$ with $\eta \in (0,1)$ sufficiently small.
This will imply
\begin{align*}
    \epsilon \sum_{k=0}^N (k+1)^2 \|\opB z_k - \opB z_{k+1/2}\|^2 & \le V_0 \le \left( p_0 L^2 + \frac{1}{2\alpha} \right) \|z_0 - z_\star\|^2 ,
\end{align*}
so that
\begin{align*}
    \sum_{k=0}^\infty (k+1)^2 \|\opB v_k - \opB v_{k+1}\|^2 \le \frac{1}{\epsilon} \left( p_0 L^2 + \frac{1}{2\alpha} \right) \|z_0 - z_\star\|^2 = \frac{C}{\alpha^2} \|z_0 - z_\star\|^2
\end{align*}
with
\begin{align*}
    C = \frac{2 - r - 2r^2 - 2r^3 + 7r^4 - 2r^6}{r(1-r)^2 (2+r)(1-r-r^2-r^3)} > 0
\end{align*}
provided that $r = \alpha L \in (0,\eta)$ for $\eta$ small enough.

\paragraph{Nonnegativity.}
Proved in the exact same way as in Lemma~\ref{lem:EAG_summability}, except that we have an additional nonnegative term $p_k L^2 \|z_k - v_k\|^2$.

\paragraph{Case $k=0$.}
We have $v_1 = z_0 - \alpha\opB v_0 = z_0 - \alpha\opB z_0$ and $z_1 = z_0 - \alpha\opB v_1$. Observe that
\begin{align}
    & V_1 + \epsilon \|\opB v_0 - \opB v_1\|^2 \nonumber \\
    & = \frac{\alpha}{1+\alpha L}\|\opB z_1\|^2 + \langle \opB z_1, z_1 - z_0 \rangle + p_1 L^2 \|z_1 - v_1\|^2 + \frac{1}{2\alpha} \|z_0 - z_\star\|^2 + \epsilon \|\opB z_0 - \opB v_1\|^2 \nonumber \\
    & = \frac{\alpha}{1+\alpha L}\|\opB z_1\|^2 + \langle \opB z_1, -\alpha \opB v_1 \rangle + p_1 L^2 \|\alpha(\opB z_0 - \opB v_1)\|^2 + \frac{1}{2\alpha} \|z_0 - z_\star\|^2 + \epsilon \|\opB z_0 - \opB v_1\|^2 \nonumber \\
    & = \frac{\alpha}{1+\alpha L}\|\opB z_1\|^2 + \langle \opB z_1, -\alpha \opB v_1 \rangle + \frac{1}{2\alpha} \|z_0 - z_\star\|^2 + \left( \alpha^2 L^2 p_1 + \epsilon \right) \|\opB z_0 - \opB v_1\|^2 . 
    \label{eqn:APS_Lyapunov_initial}
\end{align}
Because
\begin{align*}
    \alpha^2 L^2 p_1 + \epsilon = \alpha^2 L^2 \frac{\alpha}{1+\alpha L} + \frac{\alpha^2 L \left( 1 - \alpha L - \alpha^2 L^2 - \alpha^3 L^3 \right)}{2(1 + \alpha L)} = \frac{\alpha^2 L (1-\alpha^2 L^2)}{2} ,
\end{align*}
we see that the expression~\eqref{eqn:APS_Lyapunov_initial} coincides with~\eqref{eqn:EAG_C_Lyapunov_initial}, except that $\opB z_{1/2}$ is replaced with $\opB v_1$.
Therefore, the rest of the proof is exactly the same as in Lemma~\ref{lem:EAG_summability} (add the same set of inequalities to~\eqref{eqn:APS_Lyapunov_initial} with $v_1$ and $\opB v_1$ in places of $z_{1/2}$ and $\opB z_{1/2}$, respectively).

\paragraph{Case $k\ge 1$.}
Using the identities 
\begin{align*}
    z_{k+1} - z_0 & = (1-\beta_k) (z_k - z_0) - \alpha \opB v_{k+1} \\
    z_k - z_{k+1} & = \beta_k (z_k - z_0) + \alpha \opB v_{k+1} \\
    v_k - z_k & = \alpha \opB v_k - \alpha \opB v_{k-1} \\
    v_{k+1} - z_{k+1} & = \alpha \opB v_{k+1} - \alpha \opB v_k ,
\end{align*}
we begin with
\begin{align}
    & V_k - V_{k+1} \ge V_k - V_{k+1} - \frac{q_k}{\beta_k} \langle z_k - z_{k+1} , \opB z_k - \opB z_{k+1} \rangle \nonumber \\ 
    & \quad = p_k \|\opB z_k\|^2 - \frac{\alpha q_k}{\beta_k} \langle \opB z_k, \opB v_{k+1} \rangle + \frac{\alpha q_k}{\beta_k (1-\beta_k)} \langle \opB v_{k+1} , \opB z_{k+1} \rangle - p_{k+1} \|\opB z_{k+1}\|^2 \nonumber \\
    & \quad \quad + \alpha^2 p_k L^2 \|\opB v_{k-1} - \opB v_k\|^2 - \alpha^2 p_{k+1} L^2 \|\opB v_k - \opB v_{k+1}\|^2 \nonumber \\
    & \begin{aligned}
        & \quad = p_k \|\opB z_k\|^2 - \alpha k(k+1) \langle \opB z_k, \opB v_{k+1} \rangle + \alpha (k+1)^2 \langle \opB v_{k+1} , \opB z_{k+1} \rangle - p_{k+1} \|\opB z_{k+1}\|^2 \\
        & \quad \quad + \alpha^2 p_k L^2 \|\opB v_{k-1} - \opB v_k\|^2 - \alpha^2 p_{k+1} L^2 \|\opB v_k - \opB v_{k+1}\|^2 .    
    \end{aligned}
    \label{eqn:APS_monotonicity}
\end{align}
Following the idea in the prior work \cite{Tran-DinhLuo2021_halperntype} which introduced the anchored Popov's scheme, we use the following inequality:
\begin{align*}
    \|\opB v_k - \opB v_{k+1}\|^2 & = \|(\opB v_k - \opB z_k) + (\opB z_k - \opB v_{k+1})\|^2 \\
    & \le 2\|\opB v_k - \opB z_k\|^2 + 2\|\opB z_k - \opB v_{k+1}\|^2 \\
    & \le 2L^2 \|v_k - z_k\|^2 + 2\|\opB z_k - \opB v_{k+1}\|^2 \\ 
    & = 2\alpha^2 L^2 \|\opB v_k - \opB v_{k-1}\|^2 + 2\|\opB z_k - \opB v_{k+1}\|^2 ,
\end{align*}
which implies
\begin{align}
\label{eqn:APS_Lipschitz_with_parallelogram}
    2\alpha^2 L^2 \|\opB v_k - \opB v_{k-1}\|^2 + 2\|\opB z_k - \opB v_{k+1}\|^2 - \|\opB v_k - \opB v_{k+1}\|^2 \ge 0 .
\end{align}
Additionally, note the simple Lipschitzness inequality
\begin{align}
\label{eqn:APS_Lipschitz_only}
    0 & \le L^2 \|z_{k+1} - v_{k+1}\|^2 - \|\opB z_{k+1} - \opB v_{k+1}\|^2 \nonumber \\
    & = \alpha^2 L^2 \|\opB v_k - \opB v_{k+1}\|^2 - \|\opB z_{k+1} - \opB v_{k+1}\|^2 .
\end{align}
Then, given positive constants $\tau_1, \tau_2$, we can lower-bound $V_k - V_{k+1} - \epsilon (k+1)^2 \|\opB v_k - \opB v_{k+1}\|^2$ by using \eqref{eqn:APS_monotonicity} and subtracting $\frac{\tau_1}{2}\times\eqref{eqn:APS_Lipschitz_with_parallelogram}$ and $\tau_2\times\eqref{eqn:APS_Lipschitz_only}$ as
\begin{align}
    & V_k - V_{k+1} - \epsilon (k+1)^2 \|\opB v_k - \opB v_{k+1}\|^2 \nonumber \\
    & \quad \ge p_k \|\opB z_k\|^2 - \alpha k(k+1) \langle \opB z_k, \opB v_{k+1} \rangle + \alpha (k+1)^2 \langle \opB v_{k+1} , \opB z_{k+1} \rangle - p_{k+1} \|\opB z_{k+1}\|^2 \nonumber \\
    & \quad \quad + \alpha^2 p_k L^2 \|\opB v_{k-1} - \opB v_k\|^2 - \alpha^2 p_{k+1} L^2 \|\opB v_k - \opB v_{k+1}\|^2 - \epsilon (k+1)^2 \|\opB v_k - \opB v_{k+1}\|^2 \nonumber \\
    & \quad \quad - \frac{\tau_1}{2} \left( 2\alpha^2 L^2 \|\opB v_k - \opB v_{k-1}\|^2 + 2\|\opB z_k - \opB v_{k+1}\|^2 - \|\opB v_k - \opB v_{k+1}\|^2 \right) \nonumber \\
    & \quad \quad - \tau_2 \left( \alpha^2 L^2 \|\opB v_k - \opB v_{k+1}\|^2 - \|\opB z_{k+1} - \opB v_{k+1}\|^2 \right) \nonumber \\
    & \begin{aligned}
        & \quad = \alpha^2 L^2 (p_k - \tau_1) \|\opB v_{k-1} - \opB v_k\|^2 \\
        & \quad \quad + \left( \frac{\tau_1}{2} - \alpha^2 L^2 \tau_2 - \alpha^2 L^2 p_{k+1} - \epsilon (k+1)^2 \right) \|\opB v_k - \opB v_{k+1}\|^2 \\
        & \quad \quad + (p_k - \tau_1) \|\opB z_k\|^2 + \left( 2\tau_1 - \alpha k (k+1) \right) \langle \opB z_k, \opB v_{k+1} \rangle + (\tau_2 - \tau_1) \|\opB v_{k+1}\|^2 \\
        & \quad \quad + \left( \alpha (k+1)^2 - 2\tau_2 \right) \langle \opB v_{k+1}, \opB z_{k+1} \rangle + (\tau_2 - p_{k+1}) \|\opB z_{k+1}\|^2 .
    \end{aligned}
    \label{eqn:APS_Vk-Vkp1-summable_term_lowerbound}
\end{align}
We take
\begin{align*}
    \tau_1 = \frac{\alpha (k+1) \left((1+2\alpha^2 L^2)k - \alpha L (1 - 2\alpha L)\right)}{3}, \quad \tau_2 = \frac{\tau_1}{4\alpha ^2 L^2},
\end{align*}
which are positive for $\alpha L$ sufficiently small (by Lemma~\ref{lem:positive_for_small_stepsizes}), and plug in $\epsilon = \frac{\alpha^2 L \left( 1 - \alpha L - \alpha^2 L^2 - \alpha^3 L^3 \right)}{2(1 + \alpha L)}$ and
\begin{align*}
    p_k = \frac{\alpha (k+1)(k+\alpha L(k-1))}{2(1+\alpha L)}, \quad p_{k+1} = \frac{\alpha (k+2)(k+1+\alpha Lk)}{2(1+\alpha L)}
\end{align*}
to obtain
\begin{align}
\label{eqn:APS_pk_minus_tau1}
    p_k - \tau_1 = \frac{\alpha (k+1) \left( (1 + r - 4r^2 - 4r^3)k - r(1 + 2r + 4r^2) \right)}{6(1+r)}
\end{align}
and 
\begin{align*}
    \frac{\tau_1}{2} - \alpha^2 L^2 \tau_2 - \alpha^2 L^2 p_{k+1} - \epsilon (k+1)^2 = \frac{\alpha\left( T^{(2)}(r) k^2 + T^{(1)}(r) k + T^{(0)}(r) \right)}{12(1+r)}
\end{align*}
where $r = \alpha L$ and $T^{(\ell)}(r)$, $\ell=0,1,2$, are polynomials in $r$ defined by
\begin{align*}
	T^{(0)}(r) & = -r(7 + 5r - 8r^2 - 6r^3) \\
	T^{(1)}(r) & = 1 - 12r - 3r^2 + 4r^3 + 12r^4 \\
	T^{(2)}(r) & = 1 - 5r + 2r^2 + 2r^3 + 6r^4 .
\end{align*}
Therefore, by Lemma~\ref{lem:positive_for_small_stepsizes}, $\tau_1, \tau_2$ and the coefficients of the first two norm sqaure terms in \eqref{eqn:APS_Vk-Vkp1-summable_term_lowerbound} are positive provided that $r$ is sufficiently small, and thus we obtain
\begin{align*}
    & V_k - V_{k+1} - \epsilon (k+1)^2 \|\opB v_k - \opB v_{k+1}\|^2 \\
    & \quad \ge (p_k - \tau_1) \|\opB z_k\|^2 + \left( 2\tau_1 - \alpha k (k+1) \right) \langle \opB z_k, \opB v_{k+1} \rangle + (\tau_2 - \tau_1) \|\opB v_{k+1}\|^2 \\
    & \quad \quad + \left( \alpha (k+1)^2 - 2\tau_2 \right) \langle \opB v_{k+1}, \opB z_{k+1} \rangle + (\tau_2 - p_{k+1}) \|\opB z_{k+1}\|^2 \\
    & \quad = \mathrm{Tr} \left( \mathbf{M}_k \mathbf{S}_k \mathbf{M}_k^\intercal \right) ,
\end{align*}
where $\mathbf{M}_k := \begin{bmatrix}
\opB z_k & \opB v_{k+1} & \opB z_{k+1}
\end{bmatrix}$
and
\begin{gather*}
    \mathbf{S}_k = \begin{bmatrix}
    p_k - \tau_1 & \tau_1 - \frac{\alpha}{2} k(k+1) & 0 \\
    \tau_1 - \frac{\alpha}{2} k(k+1) & \tau_2 - \tau_1 & \frac{\alpha}{2} (k+1)^2 - \tau_2 \\
    0 & \frac{\alpha}{2} (k+1)^2 - \tau_2 & \tau_2 - p_{k+1}
    \end{bmatrix} = \begin{bmatrix}
    s_{11} & s_{12} & 0      \\
    s_{12} & s_{22} & s_{23} \\
    0      & s_{23} & s_{33}
    \end{bmatrix} .
\end{gather*}
We have seen in~\eqref{eqn:APS_pk_minus_tau1} that $s_{11} = p_k - \tau_1 > 0$ if $r$ is small. 
Next, we have
\begin{align*}
    t_{22} & = s_{22} - \frac{s_{12}^2}{s_{11}} \\
    & = \frac{(1-2r) (k+1) \left( (1 + 3r + 2r^2)k - r(1 - 2r) \right) \left( (1-4r^2)k - r(1 + 4r) \right) }{12Lr\left( (1 + r - 4r^2 - 4r^3)k - r(1 + 2r + 4r^2) \right)}
\end{align*}
and
\begin{align*}
    t_{33} & = s_{33} - \frac{s_{23}^2}{t_{22}} = \frac{\alpha r \left( (1 - 8r - 4r^2 - 4r^3)k - 2r(5 - 2r + 2r^2) \right)}{2 (1+r) (1-2r) \left( (1+3r+2r^2)k - r(1-2r) \right)} ,
\end{align*}
and as Lemma~\ref{lem:positive_for_small_stepsizes} shows that these quantities are positive for small $r$, we conclude that
\begin{align*}
    & V_k - V_{k+1} - \epsilon (k+1)^2 \|\opB v_k - \opB v_{k+1}\|^2 \\
    & \quad \ge \mathrm{Tr} \left( \mathbf{M}_k \mathbf{S}_k \mathbf{M}_k^\intercal \right) \\
    & \quad = \left( s_{11} \|\opB z_k\|^2 + 2s_{12} \langle \opB z_k , \opB v_{k+1} \rangle + \frac{s_{12}^2}{s_{11}} \|\opB v_{k+1}\|^2 \right) \\
    & \quad \quad + \left( t_{22} \|\opB v_{k+1}\|^2 + 2s_{23} \langle \opB v_{k+1} , \opB z_{k+1} \rangle + \frac{s_{23}^2}{t_{22}} \|\opB z_{k+1}\|^2 \right) + t_{33} \|\opB z_{k+1}\|^2 \\
    & \quad \ge 0.
\end{align*}

\end{proof}

\section{Regularity of convex and convex-concave functions}
\label{sec:regularity}

Consider an extended real-valued convex function $f\colon \reals^n \to \reals \cup \{\pm\infty\}$.
We say $f$ is \emph{proper} if $f(x) > -\infty$ for all $x\in\reals^n$ and $f(x) < +\infty$ for at least one $x\in\reals^n$.
The \emph{closure} $f$, written as $\cl f$, is defined as the constant function $\cl f \equiv -\infty$ if $f(x) = -\infty$ for some $x \in \reals^n$ and otherwise 
\begin{align*}
    \cl f (x) := \liminf_{x' \to x} f(x') .
\end{align*}
We say $f$ is \emph{closed} if $\cl f = f$ (when $f$ is proper, this is equivalent to lower semicontinuity).
If $f$ is \emph{closed, convex, and proper (CCP)}, its convex subdifferential operator $\partial f\colon \reals^n \rightrightarrows \reals^n$ defined by
\begin{align*}
    \partial f(x) = \{v \in \reals^n \,|\, f(x') \ge f(x) + \langle v, x' - x \rangle, ~\forall\, x' \in \reals^n \}
\end{align*}
is maximal monotone \cite[Corollary~31.5.2]{Rockafellar1970_convex}.
For an extended real-valued concave function $g$, we define $\cl g = -\cl (-g)$ and say it is proper (\emph{resp.\ }closed) if $-g$ is proper (\emph{resp.\ }closed) as a convex function.

Let $\vK\colon \reals^n \times \reals^m \to \reals \cup \{\pm\infty\}$ be an extended real-valued convex-concave function.
Define
\begin{align*}
    \domx \vK & = \{x\in\reals^n \,|\, \vK(x,y) < +\infty, ~\forall\, y \in \reals^m \} \\
    \domy \vK & = \{y\in\reals^n \,|\, \vK(x,y) > -\infty, ~\forall\, x \in \reals^n \} \\
    & \,\, \dom \vK = \domx \vK \times \domy \vK .
\end{align*}
We say $\vK$ is \emph{proper} if $\dom\vK \ne \emptyset$.
Define $\clx \vK(x,y)$ as the function obtained by taking the convex closure of $\vK(x,y)$ with respect to $x$ for each fixed $y$, and define $\cly \vK(x,y)$ analogously.
Then both $\clx \vK$ and $\cly \vK$ are convex-concave \cite[Corollary~33.1.1]{Rockafellar1970_convex}.
Respectively define the \emph{lower} and \emph{upper closures} of $\vK$ by
\begin{align*}
    \underline{\vK} = \clx \cly \vK , \quad  \overline{\vK} = \cly \clx \vK,
\end{align*}
which satisfy $\underline{\vK}\le \overline{\vK}$.
We say $\vK$ is \emph{closed} if $\clx \vK = \clx \cly \vK$ and $\cly \vK = \cly \clx \vK$.

The \emph{saddle subdifferential operator} $\ssd \vK\colon \reals^n \times \reals^m \rightrightarrows \reals^n \times \reals^m$ is defined by
\begin{align*}
    \ssd \vK (x,y) = \{(v,w)\in \reals^n \times \reals^m \,|\, v \in \partial_x \vK(x,y), w \in \partial_y (-\vK)(x,y) \},
\end{align*}
where $\partial_x \vK(x,y)$ denotes the convex subdifferential of $\vK(x, y)$ as a convex function of $x$ (for each fixed $y$), and $\partial_y (-\vK)(x,y)$ denotes the convex subdifferential of $-\vK(x, y)$ as a convex function of $y$ (for each fixed $x$).
If $\vK$ is \emph{closed, convex-concave, and proper (CCCP)}, then $\ssd\vK$ is maximal monotone \cite[Corollary~37.5.2]{Rockafellar1970_convex} and $(x_\star, y_\star)$ is a minimax solution to
\begin{align*}
    \underset{x \in \reals^n}{\textrm{minimize}} \,\, \underset{y \in \reals^m}{\textrm{maximize}} \,\,\,\, \vK (x,y)
\end{align*}
if and only if $0 \in \ssd\vK(x_\star, y_\star)$.
Given $z\in\reals^n \times \reals^m$, $\vK$ is differentiable at $z$ if and only if $\ssd\vK (z)$ is singleton, and if that is the case, $\ssd\vK(z) = \sop\vK(z)$ \cite[Theorem~35.8]{Rockafellar1970_convex}.

We say two convex-concave functions $\vK, \vL$ are \emph{equivalent} if $\clx \vK = \clx \vL$ and $\cly \vK = \cly \vL$.
If $\vK, \vL$ are equivalent, then they share the set of minimax solutions,
the function values at the solutions (if any) are equal, and $\ssd\vK=\ssd\vL$ \cite[Theorem~36.4, Corollary~37.4.1]{Rockafellar1970_convex}.
If $\vK$ is a closed convex-concave function, then the equivalence class of closed convex-concave functions containing $\vK$ is precisely characterized \cite[Theorem~34.2]{Rockafellar1970_convex} by
\begin{align*}
    \left\{\vL\colon \reals^n \times \reals^m \to \reals \cup \{\pm\infty\} \,|\, \vL \text{ is closed, convex-concave, and } \underline{\vK} \le \vL \le \overline{\vK} \right\} .
\end{align*}
When $f$ and $g$ are respectively CCP functions on $\reals^n$ and $\reals^m$, we define
\begin{align}
\label{eqn:f_minus_g}
    f(x) - g(y) = \begin{cases}
        f(x) - g(y) & x\in\dom f, y\in\dom g \\
        +\infty     & x\notin\dom f, y\in\dom g \\
        -\infty     & x\in\dom f, y\notin\dom g
    \end{cases}
\end{align}
but when $x\notin\dom f, y\notin\dom g$, it seems unclear how to define $f(x)-g(y)=+\infty-\infty$.
However, it turns out that one can extend~\eqref{eqn:f_minus_g} to a CCCP function on $\reals^n \times \reals^m$, and any such extensions are equivalent to one another \cite[Theorem~34.4]{Rockafellar1970_convex}.
In particular, using the convention $+\infty - \infty = -\infty$, one gets the lower closure (minimal element) of the equivalence class containing the equivalent CCCP extensions of~\eqref{eqn:f_minus_g}.
Alternatively, using the convention $+\infty - \infty = +\infty$, one gets the upper closure (maximal element) of the equivalence class.
In this regard, we can safely deal with the case $\lagrange_p (x,y) = f(x)-g(y)$ (where the equality in fact means $\lagrange_p$ belongs to the equivalence class of CCCP extensions of \eqref{eqn:f_minus_g}) without pathology or ambiguities.

\end{document}

%% file: macros.tex
\usepackage{amsmath}
\usepackage{amssymb}
\usepackage{mathtools}
\usepackage{amsthm}
\usepackage{nicefrac}
\usepackage{multirow}
\usepackage{tikz-cd} 
\usepackage{subcaption}
\usepackage{multicol,caption}
\usepackage{xspace}  

\usepackage{wrapfig}

\usepackage[flushleft]{threeparttable} 
\usepackage{colortbl}
\definecolor{bgcolor}{rgb}{0.8,1,1}
\definecolor{bgcolor2}{rgb}{0.8,1,0.8}
\definecolor{niceblue}{rgb}{0.0,0.19,0.56}

\usepackage{hyperref}
\hypersetup{colorlinks,linkcolor={blue},citecolor={niceblue},urlcolor={blue}}

\usepackage{pifont}
\definecolor{PineGreen}{RGB}{0,110,51}
\definecolor{BrickRed}{RGB}{143,20,2}

\usepackage{thmtools}

\definecolor{shadecolor}{gray}{0.9}
\declaretheoremstyle[
headfont=\normalfont\bfseries,
notefont=\mdseries, notebraces={(}{)},
bodyfont=\normalfont,
postheadspace=0.5em,
spaceabove=1pt,
mdframed={
  skipabove=8pt,
  skipbelow=8pt,
  hidealllines=true,
  backgroundcolor={shadecolor},
  innerleftmargin=4pt,
  innerrightmargin=4pt}
]{shaded}

\declaretheorem[style=shaded,within=section]{definition}
\declaretheorem[style=shaded,sibling=definition]{theorem}

\declaretheorem[style=shaded,sibling=definition]{corollary}

\declaretheorem[style=shaded,sibling=definition]{lemma}

\newcommand{\JA}{\opJ_{\alpha\opA}}
\newcommand{\JB}{\opJ_{\alpha\opB}}

\newcommand{\fA}{\mathfrak A}

\newcommand{\cut}[1]{{}}

\newcommand{\vK}{{\mathbf{K}}}
\newcommand{\vL}{{\mathbf{L}}}

\newcommand{\cN}{{\mathcal{N}}}
\newcommand{\cO}{{\mathcal{O}}}
\newcommand{\cP}{{\mathcal{P}}}

\newcommand{\cX}{{\mathcal{X}}}
\newcommand{\cY}{{\mathcal{Y}}}
\newcommand{\cZ}{{\mathcal{Z}}}

 \usepackage[bb=boondox]{mathalfa}

\usepackage{xparse}
\DeclareFontFamily{U}{ntxmia}{}
\DeclareFontShape{U}{ntxmia}{m}{it}{<-> ntxmia }{}
\DeclareFontShape{U}{ntxmia}{b}{it}{<-> ntxbmia }{}
\DeclareSymbolFont{lettersA}{U}{ntxmia}{m}{it}
\SetSymbolFont{lettersA}{bold}{U}{ntxmia}{b}{it}

\ExplSyntaxOn
\NewDocumentCommand{\varmathbb}{m}
 {
  \tl_map_inline:nn { #1 }
   {
    \use:c { varbb##1 }
   }
 }
\tl_map_inline:nn { ABCDEFGHIJKLMNOPQRSTUVWXYZ }
 {
  \exp_args:Nc \DeclareMathSymbol{varbb#1}{\mathord}{lettersA}{\int_eval:n { `#1+67 }}
 }
\exp_args:Nc \DeclareMathSymbol{varbbk}{\mathord}{lettersA}{169}
\ExplSyntaxOff

\newcommand{\TDRS}{\opT_{\textrm{DRS}}}

\DeclareMathOperator*{\argmin}{argmin}

\newcommand*{\fix}{\mathrm{Fix}\,}
\newcommand*{\zer}{\mathrm{Zer}\,}
\newcommand*{\gra}{\mathrm{Gra}\,}

\newcommand{\prox}{\mathrm{Prox}}
\newcommand{\proj}{\Pi}

\newcommand{\dom}{\mathrm{dom}\,} 
\newcommand{\domx}{\mathrm{dom}_x\,}
\newcommand{\domy}{\mathrm{dom}_y\,}

\newcommand{\reals}{\mathbb{R}}

\definecolor{lightgrey}{gray}{0.8}
\definecolor{medgrey}{gray}{0.6}
\definecolor{darkgrey}{gray}{0.4}

\newcommand{\opA}{{\varmathbb{A}}}
\newcommand{\opB}{{\varmathbb{B}}}

\newcommand{\opG}{{\varmathbb{G}}}

\newcommand{\opI}{{\varmathbb{I}}}
\newcommand{\opJ}{{\varmathbb{J}}}

\newcommand{\opT}{{\varmathbb{T}}}

\newcommand{\lagrange}{\mathbf{L}}

\newcommand{\sop}{\nabla_{\pm}}
\newcommand{\ssd}{\partial_{\pm}}

\newcommand*{\cl}{\mathrm{cl}\,}
\newcommand*{\clx}{\mathrm{cl}_x\,}
\newcommand*{\cly}{\mathrm{cl}_y\,}

\newcommand{\smeag}{SM-EAG\texttt{+}}
\newcommand{\apg}{APG\textsuperscript{*}}